\documentclass[a4paper,12pt,leqno]{amsart}
\usepackage{latexsym}
\usepackage[all]{xy}

\usepackage{amssymb} 
\usepackage{amsmath} 
\usepackage{amsthm}
\usepackage{amscd}
\usepackage{color}
\usepackage{comment}
\usepackage{mathtools}
\usepackage{mathrsfs}
\usepackage{enumerate}
\usepackage{graphicx}
\usepackage{stmaryrd}
\usepackage{comment}
\usepackage{url}

\usepackage[dvipdfmx]{hyperref}
\hypersetup{colorlinks=true}

\definecolor{gray}{gray}{0.7}
\definecolor{Gray}{gray}{0.3}

\usepackage{amscd}

\numberwithin{equation}{section}

\theoremstyle{break}
 \newtheorem{theorem}{Theorem}[section]
 \newtheorem{proposition}[theorem]{Proposition}
 \newtheorem{corollary}[theorem]{Corollary}
 \newtheorem{lemma}[theorem]{Lemma}

 \theoremstyle{definition}
 \newtheorem{definition}[theorem]{Definition}
 \newtheorem{remark}[theorem]{Remark}
 \newtheorem{example}[theorem]{Example}

\newtheorem*{acknowledgement}{Acknowledgment}

\allowdisplaybreaks[3]

\def\R{\mathbb R}
\def\Q{\mathbb Q}
\def\Z{\mathbb Z}
\def\G{\mathbb G}
\def\PP{\mathbb P}
\def\Af{\mathbb A}

\def\O{\mathcal{O}}

\def\Spec{\operatorname{Spec}}

\DeclareMathOperator{\mult}{mult}

\DeclareMathOperator{\red}{red}

\DeclareMathOperator{\Supp}{Supp}
\DeclareMathOperator{\Char}{char}

\DeclareMathOperator{\Hom}{Hom}

\DeclareMathOperator{\Pic}{Pic}
\DeclareMathOperator{\Cl}{Cl}
\DeclareMathOperator{\Ker}{Ker}
\DeclareMathOperator{\Image}{Im}
\DeclareMathOperator{\Gr}{Gr}

\begin{document}
\title[Vanishing theorems and adjoint linear systems on normal surfaces]{Vanishing theorems and adjoint linear systems on normal surfaces in positive characteristic}
\author [M. Enokizono]{Makoto Enokizono}
\address{Department of Mathematics, Faculty of Science and Technology, 
Tokyo University of Science, 
2641 Yamazaki, Noda, Chiba 278-8510, Japan}
\email{\url{enokizono_makoto@ma.noda.tus.ac.jp}}


\begin{abstract}
We prove the Kawamata--Viehweg vanishing theorem for a large class of divisors on surfaces in positive characteristic.
By using this vanishing theorem, Reider-type theorems and extension theorems of morphisms for normal surfaces are established.
As an application of the extension theorems, we characterize non-singular rational points on any plane curve over an arbitrary base field in terms of rational functions on the curve.
\end{abstract}

\maketitle

\setcounter{tocdepth}{1}

\tableofcontents

\section{Introduction}
\label{sec:Introduction}

This paper is a continuation of \cite{Eno}.
The purpose of this paper is to establish the vanishing theorem and the criterion for spannedness of adjoint linear systems $|K_X+D|$ for ``positive'' divisors $D$ on normal surfaces $X$ in positive characteristic.

\subsection{Vanishing theorems}

Kodaira-type vanishing theorem is a fundamental tool in algebraic geometry.
Unfortunately, the Kodaira vanishing theorem fails in positive characteristic \cite{Ray2}.
In the case of surfaces, it is known that Kodaira's vanishing or more generally, the Kawamata--Viehweg vanishing for $\Z$-divisors holds except for quasi-elliptic surfaces with Kodaira dimension $1$ and surfaces of general type (\cite{S-B}, \cite{Ter}, \cite{Muk}).
For these exceptional surfaces, Kodaira's vanishing also holds under some positive condition for divisors $D$ (e.g., the self-intersection number $D^2$ is large to some extent) (\cite{Fuj}, \cite{S-B}, \cite{Ter}, \cite{CeFa}, \cite{Zha}).
However for the Kawamata--Viehweg vanishing for $\Q$-divisors, there exist counter-examples even for smooth rational surfaces (\cite{CaTa}, \cite{Ber}).
Under some liftability conditions, it is known that the Kawamata--Viehweg vanishing holds in full generality (cf.\ \cite{Hara}, \cite{Lan}).
But for an arbitrary surface, only known results are the asymptotic versions of the Kawamata--Viehweg vanishing (\cite{Tana}).
One of the main result in this paper is the following vanishing theorem for surfaces in positive characteristic:

\begin{theorem}[Theorem~\ref{bigzposvan}]\label{Introbigzposvan} 
Let $X$ be a normal proper surface over an algebraically closed field $k$ of positive characteristic.
Let $D$ be a big $\Z$-positive divisor on $X$.
If $\dim|D|\ge \dim H^{1}(\O_X)_n$, then $H^i(\O_{X}(K_X+D))=0$ holds for any $i>0$.
\end{theorem}

Here $H^{1}(\O_X)_n$ denotes the nilpotent part of $H^1(\O_X)$ under the Frobenius action and a divisor $D$ on $X$ is said to be {\em $\Z$-positive} if $B-D$ is not nef over $B$ for any effective negative definite divisor $B>0$ on $X$.
Typical examples of $\Z$-positive divisors are the round-ups $D=\ulcorner M\urcorner$ of nef $\Q$-divisors $M$ and numerically connected divisors which are not negative definite. 
Theorem~\ref{Introbigzposvan} is new even when $X$ is smooth (and of general type) and $D$ is ample.

A well-known proof of the Kawamata--Viehweg vanishing theorem is to use the covering method and reduce to the Kodaira vanishing theorem (cf.\ \cite{KMM}).
In positive characteristic, although the Kodaira vanishing holds for almost all surfaces not of general type, it is difficult to apply the covering method.
The reason is that the covering method reduces the vanishing of the cohomology on a given surface $X$ to that on the total space $Y$ of a covering $Y\to X$ 
but in many cases, $Y$ must be of general type and so Kodaira's vanishing can not be applied.
For this, we use another method to prove Theorem~\ref{Introbigzposvan}.
The key observation to prove Theorem~\ref{Introbigzposvan} is to study the connectedness of effective divisors and to prove the following lemma:

\begin{lemma}[Corollary~\ref{bigzposconn}]
Let $D>0$ be an effective big $\Z$-positive divisor on a normal complete surface $X$ over an algebraically closed field $k$.
Then $H^0(\O_D)\cong k$ holds.
\end{lemma}

For the higher dimensional case, we will prove the following vanishing theorem on $H^1$: 

\begin{theorem}[Theorem~\ref{Miyvanthm}]   
Let $X$ be a normal projective variety of dimension greater than $1$ over an algebraically closed field $k$.
Let $D$ be a divisor on $X$
such that $D=\ulcorner M\urcorner$ for some nef $\R$-divisor $M$ on $X$
with $\kappa(D)\ge 2$ or $\nu(M)\ge 2$.
If $\Char k>0$, we further assume that $\dim |D|\ge \dim H^1(\O_X)_n$.
Then $H^1(X,\O_X(-D))=0$ holds.
\end{theorem}

\subsection{Adjoint linear systems}
For adjoint linear systems $|K_X+D|$ of divisors $D$ on varieties $X$, it is expected that the ``positivity'' of the divisor $D$ implies the ``spannedness'' of the linear system $|K_X+D|$
(Fujita's conjecture is a typical one).
For smooth surfaces $X$ in characteristic $0$, Reider's theorem \cite{Rei} roughly says that if the adjoint linear system $|K_X+D|$ for a nef and big divisor $D$ has a base point, there exists a curve $B$ on $X$ obstructing the basepoint-freeness such that $D$ and $B$ satisfy some numerical conditions.
Reider's method enables us to give various applications of adjoint linear systems, especially, the affirmative answer to Fujita's conjecture for surfaces in characteristic $0$.
On the other hand, although Shepherd-Barron and others (\cite{S-B}, \cite{Ter}, \cite{Mor}, \cite{CeFa}) studied adjoint linear systems on smooth surfaces in positive characteristic by using Reider's method based on some Bogomolov-type inequalities,
 there exist counter-examples to Fujita's conjecture for surfaces in positive characteristic (\cite{GZZ}).
In this paper, as applications of Theorem~\ref{Introbigzposvan} and other vanishing results (Corollary~\ref{Francia}, Propositions~\ref{surfalphavan} and \ref{surfpencilvan}), we give some results for adjoint linear systems on not necessarily smooth surfaces in positive characteristic.
Here we state immediate corollaries of the main result (Theorem~\ref{ReithmI}):

\begin{corollary}[Corollary~\ref{bpfcor}]
Let $X$ be a normal complete surface over an algebraically closed field $k$.
When $\Char k>0$, we further assume that the Frobenius map on $H^1(\O_X)$ is injective.
Let $x\in X$ be at most a rational singularity.
Let $L$ be a divisor on $X$ which is Cartier at $x$.
We assume that there exists an integral curve $D\in |L-K_X|$ passing through $x$ such that $(X,x)$ or $(D,x)$ is singular, and $D$ is analytically irreducible at $x$ when $\Char k>0$.
Then $x$ is not a base point of $|L|$.
\end{corollary}

\begin{corollary}[Corollary~\ref{bpfcor2}]
Let $X$ be a normal proper surface with at most singularities of geometric genera $p_g\le 3$ over an algebraically closed field $k$ of positive characteristic.
Let $D$ be a nef divisor on $X$ such that $K_X+D$ is Cartier.
Then $|K_X+D|$ is base point free if the following three conditions hold:
\smallskip

\noindent
$(\mathrm{i})$ $D^2>4$,

\smallskip

\noindent
$(\mathrm{ii})$ $DB\ge 2$ for any curve $B$ on $X$, and

\smallskip

\noindent
$(\mathrm{iii})$
$\dim |D|\ge \dim H^1(\O_X)_n+3$.
\end{corollary}

The following is a partial answer to Fujita's conjecture for surfaces in positive characteristic:

\begin{corollary}[Corollary~\ref{Fujitacor}]
Let $X$ be a projective surface with at most rational double points over an algebraically closed field $k$ of positive characteristic.
Let $H$ be an ample divisor on $X$.
Then $|K_X+mH|$ is base point free for any $m\ge 3$ with $\dim|mH|\ge \dim H^1(\O_X)_n+3$, and is very ample for any $m\ge 4$ with $\dim |mH|\ge \dim H^1(\O_X)_n+6$.
\end{corollary}

For other corollaries (e.g., for the pluri-(anti)canonical systems on normal surfaces), see Section~5.

\subsection{Extension theorems}

In this paper, ``extension theorem'' means the result of the extendability of morphisms defined on a divisor to the whole variety.
For the surface case, the extension theorem goes back to the results of Saint-Donat \cite{S-D} and Reid \cite{Reid} for $K3$ surfaces.
After that, Serrano \cite{Ser} and Paoletti \cite{Pao} proved extension theorems for integral curves on smooth surfaces.
These results were generalized in \cite{Eno} for possibly reducible or non-reduced curves on normal surfaces in characteristic $0$.
In this paper, as an application of the Reider-type theorem (Theorem~\ref{ReithmII}), we give a positive characteristic analog of this extension theorem:

\begin{theorem}[Theorem~\ref{extnthm}] 
Let $D>0$ be an effective divisor on a normal complete surface $X$ over an algebraically closed field $k$ of positive characteristic
and assume that any prime component $D_i$ of $D$ has positive self-intersection number.
Let $\varphi\colon D\to \mathbb{P}^1$ be a finite separable morphism of degree $d$.
If $D^2>\mu(q_X,d)$ and $\dim |D|\ge 3d+\dim H^1(\O_X)_n$, then
there exists a morphism $\psi\colon X\to \mathbb{P}^1$ such that $\psi|_D=\varphi$.
\end{theorem}
For the definition of $\mu(q_X,d)$, see Definition~\ref{qdef}
(e.g., $\mu(q_X,d)\le (d+1)^2$ holds when $X$ is smooth).
Some variants of extension theorems are established in Section~6.

As an application of the extension theorems, we give a characterization of non-singular $k$-rational points of plane curves $D\subset \PP^2$ over any base field $k$ in terms of rational functions on $D$, which is a natural generalization of the classical result \cite{Nam} that 
the gonality of smooth complex plane curves of degree $m$ is equal to $m-1$:

\begin{theorem}[Theorem~\ref{planethm}]
Let $D\subset \PP^2$  be a plane curve of degree $m\ge 3$ over an arbitrary base field $k$.
Then there is a one-to-one correspondence between 

\smallskip

\noindent
$(\mathrm{i})$
the set of non-singular $k$-rational points of $D$ which are not strange, and 

\smallskip

\noindent
$(\mathrm{ii})$
the set of finite separable morphisms $D\to \PP^1$ of degree $m-1$ up to automorphisms of $\PP^1$.

\smallskip

\noindent
Moreover, any finite separable morphism $D\to \PP^1$ has degree greater than or equal to $m-1$.
\end{theorem}

\subsection{Structure of the paper}
The present paper is organized as follows.
In Section~2, we fix some notations and terminology used in this paper.
In Section~3, we discuss chain-connected divisors, which play a central role in this paper.
The key result in this section is the chain-connectedness of big $\Z$-positive divisors (Proposition~\ref{bigzposcc}).
This is used to prove the main vanishing theorem (Theorem~\ref{bigzposvan}).
In the first half of Section~4, we study the kernel $\alpha(X,D)$ of the restriction map $H^1(\O_X)\to H^1(\O_D)$ for divisors $D$ on $X$ 
following the arguments in \cite{Mum2}, \cite{Fra} and \cite{BHPV}.
The rest of Section~4 is devoted to the vanishing theorem on surfaces in positive characteristic and its generalization.
The essential idea is the combination of Fujita's and Mumford's arguments ({\cite[(7.4)~Theorem]{Fuj}}, {\cite[p.99]{Mum2}}) and the chain-connectedness of big $\Z$-positive divisors.
In Section~5, we study adjoint linear systems on normal surfaces in positive characteristic as an application of the vanishing theorems obtained in Section~4.
The proof of the main result (Theorem~\ref{ReithmI}) is almost similar to that of Theorem~5.2 in \cite{Eno}.
The only difference is to use the chain-connected component decomposition (cf.\ Corollary~1.7 in \cite{Kon}) instead of the integral Zariski decomposition (Theorem~3.5 in \cite{Eno}).
In Section~6, we give extension theorems for normal surfaces in positive characteristic by using the Reider-type theorem (Theorem~\ref{ReithmII}) obtained in Section~5.
As an application of the extension theorems,
non-singular $k$-rational points of any plane curve $D\subset \PP^2$ over an arbitrary base field $k$ are characterized in terms of rational functions on $D$ in Section~7.
In Appendix~A, the Mumford's intersection form on a normal projective variety is formulated.

\begin{acknowledgement}
The author is grateful to Professor Kazuhiro Konno for valuable discussions and for warm encouragements.
He would like to thank Hiroto Akaike, Sho Ejiri, Tatsuro Kawakami, Yuya Matsumoto and Shou Yoshikawa for helpful discussions and answering his questions.
He was partially supported by JSPS Grant-in-Aid for Research Activity Start-up: 19K23407 and JSPS Grant-in-Aid for Young Scientists: 20K14297. 
\end{acknowledgement}

\section{Notations and terminology}
\label{sec:Notations and terminology}

\begin{itemize}
\item
In this paper, we mainly work on the category of algebraic schemes over a field $k$.

\item
A {\em divisor} means a Weil divisor (not necessarily $\Q$-Cartier).

\item 
For a $\Q$-divisor (resp.\ nef $\R$-divisor) $D$ on a normal proper variety $X$, we denote by $\kappa(D)$ (resp.\ $\nu(D)$) the {\em Iitaka dimension} (resp.\ {\em numerical dimension}) of $D$ (for the details, see Chapter~6 in \cite{KMM} or Section~2.4 in \cite{Fuj}).

\item
For a normal projective variety $X$ of dimension $\ge 2$ over an infinite field $k$ and a divisor $D$ on $X$, we freely use the following Bertini-type result (for the details, see Section~1.1 in \cite{HuLe}):
Any general hyperplane $Y$ on $X$ is also a normal projective variety and satisfies $\O_{Y}(D|_Y)\cong \O_X(D)|_Y$.

\item
We freely use Mumford's intersection form (\cite{Mum}) (for the higher dimensional case, see Appendix~A).

\item
For a $p$-linear transform $F\colon V\to V$ of a finite dimensional vector space $V$ over a field $k$ of characteristic $p>0$,
we write the {\em semi-simple part of $V$} (resp.\ {\em nilpotent part of $V$}) by $V_s:=\Image F^l$ (resp.\ $V_n:=\Ker F^l$), $l\gg 0$.
Then it is well known that $V=V_s\oplus V_n$,
and there exists a $k$-basis $\{e_i\}$ of $V_s$ such that $F(e_i)=e_i$ for each $i$ when $k$ is algebraically closed.

\item
A finite surjective morphism $f\colon X\to Y$ from a proper scheme $X$ to a variety $Y$ is called {\em separable} if the restriction $f|_{X_i}$ induces a separable field extension $(f|_{X_{i,\red}})^{*}\colon K(Y)\hookrightarrow K(X_{i,\red})$ between function fields for each irreducible component $X_i$ of $X$.
\end{itemize}

\section{Chain-connected divisors}
\label{sec:Chain-connected divisors}

\subsection{Connectedness of effective divisors}
First, we introduce some notions about connectivity for effective divisors on normal surfaces, which is well known for smooth surfaces.
In this section, $X$ stands a normal proper surface over a base field $k$ (or a normal compact analytic surface) unless otherwise stated.

\begin{definition}[Connectedness for effective divisors]
(1) Let $D>0$ be an effective divisor on $X$.
Then $D$ is called {\em chain-connected} (resp.\ {\em numerically connected}) if $-A$ is not nef over $B$ (resp.\ $AB>0$) for any effective decomposition $D=A+B$ with $A, B>0$.

\smallskip

\noindent
(2) Let $m\in \Q$ be a rational number.
Then $D$ is called {\em $m$-connected} (resp.\ {\em strictly $m$-connected}) if $AB\ge m$ (resp.\ $AB>m$) for any effective decomposition $D=A+B$ with $A,B>0$.
Clearly, numerical connectivity is equivalent to strict $0$-connectivity and implies chain-connectivity.

\smallskip

\noindent
(3) For an effective divisor $D$ and a subdivisor $0<D_0\le D$, a {\em connecting chain from $D_0$ to $D$} is defined to be a sequence of subdivisors $D_0<D_1<\dots <D_m=D$ such that $C_i:=D_i-D_{i-1}$ is prime and $D_{i-1}C_i>0$ for each $i=1,\ldots, m$.
We regard $D_0=D$ as a connecting chain from $D$ to $D$ ($m=0$ case).
\end{definition}

The following lemma is easy and well-known. 
\begin{lemma} \label{ampleconn}
Let $H$ be an ample Cartier divisor on $X$ and $n>1$.
Then any effective divisor $D$ which is numerically equivalent to $nH$ is $(n-\frac{1}{H^2})$-connected.
\end{lemma}

\begin{proof}
For the readers' convenience, we give a proof.
Let $D=A+B$ be a non-trivial effective decomposition.
By the Hodge index theorem, we can write $A\equiv aH+A'$ and $B\equiv bH+B'$, 
where $a:=AH/H^2$, $b:=BH/H^2$, $A'H=0$ and $B'H=0$ with $A'^{2}\le 0$ and $B'^{2}\le 0$.
Note that both $a$ and $b$ are greater than or equal to $1/H^2$ since $H$ is ample Cartier.
Since $A+B\equiv nH$, it follows that $a+b=n$ and $A'+B'\equiv 0$.
Thus we have 
$$
AB=abH^2+A'B'=a(n-a)H^2-A'^2\ge \frac{1}{H^2}\left(n-\frac{1}{H^2}\right)H^2=n-\frac{1}{H^2}.
$$
\end{proof}

Similarly, we can prove the following (cf.\ {\cite[Lemma~2]{Ram}} or {\cite[p.242]{KaMa}}).
\begin{lemma}[Ramanujam's connectedness lemma] \label{nefbignumconn}
Let $D$ be an effective, nef and  big divisor on $X$.
Then $D$ is numerically connected.
\end{lemma}

The following lemmas are due to \cite{Kon} (the proof also works on possibly singular normal surfaces).

\begin{lemma}[\cite{Kon} Proposition~1.2]\label{ccchar}
Let $D>0$ be an effective divisor on $X$.
Then the following are equivalent.

\smallskip

\noindent
$(\mathrm{i})$
$D$ is chain-connected.

\smallskip

\noindent
$(\mathrm{ii})$
For any subdivisor $0<D_0\le D$, there exists a connecting chain from $D_0$ to $D$.

\smallskip

\noindent
$(\mathrm{iii})$
There exist a prime component $D_0$ of $D$ and a connecting chain from $D_0$ to $D$.
\end{lemma}

\begin{lemma}[\cite{Kon} Proposition~1.5 (3)] \label{ccc}
Let $D>0$ be an effective divisor on $X$ with connected support.
Then there exists the greatest chain-connected subdivisor $0<D_{\mathrm{c}}\le D$ such that $\mathrm{Supp}(D_{\mathrm{c}})=\mathrm{Supp}(D)$ and $-D_{\mathrm{c}}$ is nef over $D-D_{\mathrm{c}}$.
\end{lemma}

\begin{definition}[Chain-connected component]
Let $D>0$ be an effective divisor on $X$ with connected support.
The greatest chain-connected subdivisor of $D$ is called the {\em chain-connected component} of $D$ and denoted by $D_{\mathrm{c}}$.
Similarly, for any effective divisor $D>0$ on $X$, we can take the chain-connected component for each connected component of $D$.
\end{definition}

\begin{proposition}\label{ccpullback}
Let $\pi\colon X'\to X$ be a proper birational morphism between normal complete surfaces.
Then for any chain-connected divisor $D$ on $X$, the round-up of the Mumford pull-back $\ulcorner \pi^{*}D \urcorner$ is chain-connected.
\end{proposition}

\begin{proof}
The proof is similar to that of Proposition~3.19 in \cite{Eno}.
We write $D':=\ulcorner \pi^{*}D \urcorner=\pi^{*}D+D_{\pi}$, where $D_{\pi}$ is a $\pi$-exceptional $\Q$-divisor on $X'$ with $\llcorner D_{\pi} \lrcorner=0$.
Note that $D'$ has connected support since so does $D$.  
Assume that $D'$ is not chain-connected, that is, $D'_{\mathrm{c}}<D'$.
Let $B':=D'-D'_{\mathrm{c}}$.
It follows from Lemma~\ref{ccc} that $\mathrm{Supp}(D'_{\mathrm{c}})=\mathrm{Supp}(D')$ and $-D'_{\mathrm{c}}$ is nef over $B'$.
We write $B'=\pi^{*}B+B_{\pi}$, where $B:=\pi_{*}B'\ge 0$ and $B_{\pi}$ is a $\pi$-exceptional $\Q$-divisor.
Let $B_{\pi}-D_{\pi}=G^{+}-G^{-}$ be the decomposition of effective $\pi$-exceptional $\Q$-divisors $G^{+}$ and $G^{-}$ having no common components.
Then the support of $G^{+}$ is contained in that of $B'$.
First we assume that $G^{+}>0$.
Then there exists a prime component $C$ of $G^{+}$ such that $G^{+}C<0$ since $G^{+}$ is negative definite.
It follows that $C\le B'$ and
$$
-D'_{\mathrm{c}}C=(B'-D')C=(\pi^{*}(B-D)+G^{+}-G^{-})C=(G^{+}-G^{-})C<0,
$$
which contradicts the nefness of $-D'_{\mathrm{c}}$ over $B'$.
Hence we have $G^{+}=0$.
This and $\llcorner D_{\pi} \lrcorner=0$ imply $B>0$.
Since $D$ is chain-connected and $D-B=\pi_{*}D'_{\mathrm{c}}>0$, there exists a prime component $C\le B$ such that $(B-D)C<0$.
Let $\widehat{C}$ be the proper transform of $C$ on $X'$.
Then we have $\widehat{C}\le B'$ and 
$$
-D'_{\mathrm{c}}\widehat{C}=(B'-D')\widehat{C}=(\pi^{*}(B-D)-G^{-})\widehat{C}=(B-D)C-G^{-}\widehat{C}<0,
$$
which contradicts the nefness of $-D'_{\mathrm{c}}$ over $B'$.
Hence we conclude that $D'$ is chain-connected.
\end{proof}

\begin{example}
The numerically connected version of Proposition~\ref{ccpullback} does not hold:
Let $f\colon S\to \PP^1$ be an elliptic surface having a singular fiber $F=f^{-1}(0)$ of type $\mathrm{I}_1$.
Take three blow-ups $X':=S_3\xrightarrow{\rho_3} S_2\xrightarrow{\rho_2} S_1\xrightarrow{\rho_1} S_0:=S$ at single points $p_i\in S_{i-1}$,
where $p_1$ and $p_2$ respectively are a general point in $F$ and the node of $F$ and $p_3$ is a point in the intersection of the $\rho_2$-exceptional curve and the proper transform of $F$.
Let $C'_1$, $C'_{2}$ and $C'_{3}$ respectively denote the $\rho_3$-exceptional curve, the proper transform of the $\rho_2$-exceptional curve and the proper transform of $F$ on $X'$.
Then we have $(C'_{i})^2=-i$ and $C'_{i}C'_{j}=1$ for $i\neq j$.
Let $\pi\colon X'\to X$ be the contraction of $C'_{3}$ and put $C_{i}:=\pi_{*}C'_{i}$ for $i=1,2$.
Note that $X$ has one cyclic quotient singularity at $\pi(C'_{3})$ and $C'_i=\pi^{*}C_i-(1/3)C_3$ for $i=1,2$.
Thus we have $C_1^2=-2/3$, $C_2^2=-5/3$ and $C_1C_2=4/3$.
Then $D:=2C_1+2C_2$ is numerically connected but $\ulcorner \pi^{*}D \urcorner=2C'_1+2C'_2+2C'_3$ is not numerically connected since $(C'_1+C'_2+C'_3)^2=0$.
\end{example}

\begin{lemma}\label{ccconn}
Let $D>0$ be a chain-connected divisor on $X$.
Then $H^0(\O_D)$ is a field.
Moreover, if $D$ contains a prime divisor $C$ such that $H^0(\O_X)\cong H^0(\O_C)$, then we have $H^0(\O_D)\cong H^0(\O_X)$.
\end{lemma}

\begin{proof}
First, we show the claim when $X$ is regular.
By Lemma~\ref{ccchar}, we can take a connecting chain $C=:D_0<D_1<\cdots<D_m:=D$ for any prime component $C$ of $D$.
Putting $C_i:=D_i-D_{i-1}$, we have $D_{i-1}C_i>0$ for each $i$.
By the exact sequence
$$
0\to \O_{C_i}(-D_{i-1})\to \O_{D_i}\to \O_{D_{i-1}}\to 0
$$
and $H^0(\O_{C_i}(-D_{i-1}))=0$ for each $i$, we have a chain of injections
$$
H^0(\O_X)\hookrightarrow H^0(\O_D) \hookrightarrow H^0(\O_{D_{m-1}}) \hookrightarrow \cdots \hookrightarrow H^0(\O_{D_0}).
$$
Thus $H^0(\O_D)$ is a subfield of $H^0(\O_C)$.
If moreover $H^0(\O_X)\cong H^0(\O_C)$, then all the injections above are isomorphisms.

For a general $X$, we take a resolution $\pi\colon X'\to X$ and put $D':=\ulcorner \pi^{*}D \urcorner$.
By Proposition~\ref{ccpullback}, $D'$ is also chain-connected.
We note that there are natural injections $H^0(\O_X)\hookrightarrow H^0(\O_D)\hookrightarrow H^0(\O_{D'})$.
Thus $H^0(\O_D)$ is a subfield of $H^0(\O_{D'})$.
If $D$ contains a prime divisor $C$ with $H^0(\O_X)\cong H^0(\O_C)$, 
then the proper transform $\widehat{C}$ of $C$ satisfies $H^0(\O_{X'})\cong H^{0}(\O_{\widehat{C}})$ since $H^0(\O_{X})\cong H^0(\O_{X'})$ and $H^0(\O_{C})\cong H^0(\O_{\widehat{C}})$.
Thus we obtain $H^0(\O_{D'})\cong H^0(\O_{X'})$ by the assertion for regular surfaces.
Hence we conclude that $H^0(\O_X)\cong H^0(\O_D)$.
\end{proof}

\subsection{Chain-connected v.s. $\Z$-positive}

Let us recall the notion of $\Z$-positive divisors on normal surfaces introduced in \cite{Eno}.

\begin{definition}[$\Z$-positive divisors]
A divisor $D$ on a normal complete surface $X$ is called {\em $\Z$-positive}
if $B-D$ is not nef over $B$ for any effective negative definite divisor $B>0$ on $X$.
\end{definition}

Typical examples of $\Z$-positive divisors are chain-connected divisors which are not negative definite and the round-ups of nef $\R$-divisors (cf.\ Proposition~3.16 in \cite{Eno}).
In this subsection, we prove the following result which is an analog of Lemma~\ref{nefbignumconn}:

\begin{proposition} \label{bigzposcc}
Any effective big $\Z$-positive divisor is chain-connected.
\end{proposition}

\begin{proof}
Let $D>0$ be an effective big $\Z$-positive divisor on $X$.
First we note that the support of $D$ is connected since $D$ is obtained by a connecting chain from the round-up $\ulcorner P(D)\urcorner$ of the positive part $P(D)$ in the Zariski decomposition of $D$ (cf.\ Proposition~3.16 in \cite{Eno}) and the support of $P(D)$ is connected by Lemma~\ref{nefbignumconn}.
Let $D_{\mathrm{c}}$ be the chain-connected component of $D$ 
and suppose that $D-D_{\mathrm{c}}\neq 0$.
Let us take its Zariski decomposition $D-D_{\mathrm{c}}=P+N$.
If $P=0$, then it follows that $D-D_{\mathrm{c}}$ is negative definite, which contradicts the $\Z$-positivity of $D$.
Thus we have $P\neq 0$.
Since $D_{\mathrm{c}}P\le 0$, $\Supp(D_{\mathrm{c}})=\Supp(D)$ and $P$ is nef, 
it follows that $P$ is numerically trivial over $D$.
In particular, we have $P^2=0$.
Since $\Supp(D)$ is connected and $P\neq 0$, the support of $P$ coincides with that of $D$.
Thus $D$ is negative semi-definite, which contradicts the bigness of $D$ (since $P(D)^2>0$).
Hence we conclude that $D=D_{\mathrm{c}}$.
\end{proof}

\begin{remark}
(1) Conversely, any big chain-connected divisor is $\Z$-positive since for an effective divisor $D>0$ with connected support, $D$ is big if and only if $D$ is not negative semi-definite (cf.\ Lemma~A.12 in \cite{Eno}).
Thus for any effective big divisor $D$ with connected support, its $\Z$-positive part $P_{\Z}$ in the integral Zariski decomposition (Theorem~3.5 in \cite{Eno}) coincides with the chain-connected component $D_{\mathrm{c}}$.

\smallskip

\noindent
(2) In general, effective $\Z$-positive divisors are not chain-connected even if $D$ has connected support. 
For example, a multiple $nF$ of a fiber $F=f^{-1}(t)$ of a fibration $f\colon X\to B$ over a curve $B$ is $\Z$-positive but not chain-connected for $n\ge 2$.
\end{remark}

Combining Proposition~\ref{bigzposcc} with Lemma~\ref{ccconn}, we obtain the following.

\begin{corollary} \label{bigzposconn}
Let $D$ be an effective big $\Z$-positive divisor on a normal complete surface.
Then $H^0(\O_D)$ is a field.
\end{corollary}

\subsection{Base change property}
In this subsection, let $X$ be a normal proper geometrically connected surface over a field $k$ and $k\subset k'$ a separable field extension.
Then $X_{k'}:=X\times _{k}k'$ is also a normal surface with $H^0(\O_{X_{k'}})\cong k'$.

\begin{lemma} \label{zposbasechange}
Let $D$ be a pseudo-effective $\Z$-positive divisor on $X$.
Then $D_{k'}$ is also a pseudo-effective $\Z$-positive divisor on $X_{k'}$.
\end{lemma}

\begin{proof}
Let $D=P+N$ be the Zariski decomposition of $D$.
Then there exists a connecting chain $D_0:=\ulcorner P\urcorner <D_1<\cdots <D_N:=D$ such that $C_i:=D_i-D_{i-1}$ is prime and satisfies $D_{i-1}C_i>0$ for each $i$ from Proposition~3.16 in \cite{Eno}.
On the other hand, one can see that $D_{k'}=P_{k'}+N_{k'}$ is also the Zariski decomposition of $D_{k'}$ (Indeed, $P_{k'}$ is also nef and $P_{k'}N_{k'}=PN=0$. Thus the Hodge index theorem implies that $N_{k'}$ is negative definite if $D$ is big.
When $D$ is not big, then by the pseudo-effectivity of $D$, this can be written by the limit of big divisors in the numerical class group of $X$.
Thus the claim holds by the continuity of the Zariski decomposition).
Since the extension $k'/k$ is separable, it follows that $\llcorner N_{k'}\lrcorner=\llcorner N\lrcorner_{k'}=\sum_{i=1}^{N}C_{i,k'}$ and $C_{i,k'}$ is reduced.
Let $C_{i,k'}=\sum_{j=1}^{l(i)}C'_{i,j}$ be the irreducible decomposition.
Since $D_{i-1,k'}C'_{i,j}>0$ for any $i$ and $j$, we can construct a connecting chain from $\ulcorner P_{k'}\urcorner$ to $D_{k'}$, whence $D_{k'}$ is $\Z$-positive from Proposition~3.16 in \cite{Eno}.
\end{proof}

Combining Lemma~\ref{zposbasechange} with Proposition~\ref{bigzposcc}, we obtain the following:
\begin{corollary}
Let $D>0$ be a chain-connected divisor on $X$ which is not negative semi-definite.
Then $D_{k'}$ is also chain-connected.
\end{corollary}

\begin{remark}
In general, the chain-connectivity is not preserved by a separable base change.
Indeed, let $X$ be a surface obtained by the blow-up of $\PP^2$ at a closed point $x$ such that the extension $k(x)/k$ is non-trivial and separable.
Then the exceptional divisor $E_x$ on $X$ is chain-connected but $E_{x,k(x)}$ is a disjoint union of the exceptional curves on $X_{k(x)}$.
\end{remark}

\begin{corollary} \label{bigzposinsep}
Let $D$ be an effective big $\Z$-positive divisor on $X$.
Then $H^0(\O_D)/k$ is a purely inseparable field extension.
\end{corollary}

\begin{proof}
Let $k'$ be the separable closure of $k$ in the field $H^0(\O_D)$ and assume $k\neq k'$.
Then $H^0(\O_{D_{k'}})\cong H^0(\O_{D})\otimes_{k}k'$ is not a field, which contradicts Lemma~\ref{zposbasechange} and Corollary~\ref{bigzposconn}.
\end{proof}

\subsection{Chain-connected divisors on projective varieties}

Now we introduce chain-connected divisors on higher dimensional varieties.

\begin{definition}[Chain-connected divisors]
Let $X$ be a normal projective variety of dimension $n\ge 3$ over an algebraically closed field $k$ and $H$ an ample divisor on $X$.
A divisor $D$ on $X$ is called {\em $H$-nef} (resp.\ {\em $H$-nef over an effective divisor $B$}) if $H^{n-2}AC\ge 0$ for any prime divisor $C$ (resp.\ any prime component $C\le B$) on $X$, where we use Mumford's intersection form for $n-2$ Cartier divisors and two Weil divisors on $X$ (see Appendix~A).
An effective divisor $D>0$ on $X$ is said to be {\em chain-connected with respect to $H$} if $-A$ is not $H$-nef over $B$ for any non-trivial effective decomposition $D=A+B$.
We simply call $D$ {\em chain-connected} if it is chain-connected with respect to some ample divisor $H$ on $X$.
\end{definition}

\begin{example} \label{cconnex}
The following effective divisors $D$ are typical examples of chain-connected divisors:

\smallskip

\noindent
(i) $D$ is reduced and connected.

\smallskip

\noindent
(ii) $D=\ulcorner M \urcorner$, where $M$ is an $\R$-divisor which is nef in codimension $1$ and satisfies $\nu(M)\ge 2$ or $\kappa(D)\ge 2$.

\smallskip

\noindent
Here, we say that an $\R$-Cartier $\R$-divisor $D$ on a normal complete variety $X$ is {\em nef in codimension $1$} if there exists a closed subset $Z$ of $X$ with codimension $\ge 2$ such that $DC\ge 0$ holds for any integral curve $C$ on $X$ not contained in $Z$ (which is called numerically semipositive in codimension one in \cite{Fujita}).
A Weil $\R$-divisor $D$ on $X$ is called {\em nef} (resp.\ {\em nef in codimension $1$}) if there exist an alteration $\pi\colon X'\to X$ and a nef (resp.\ nef in codimension $1$) $\R$-Cartier $\R$-divisor $D'$ on $X'$ such that $D=\pi_{*}D'$.
By using Mumford's intersection form, $\nu(D)\ge 2$ makes sense as $H^{\dim X-2}D^2>0$ for some ample divisor $H$ on $X$.
Note that for $\dim X=2$, the condition (ii) is equivalent to that $D$ is the round-up of a nef and big $\R$-divisor $M$. 
Therefore, the proof for the chain-connectivity of the divisor $D$ in (ii) is reduced to the surface case by cutting with general hyperplanes in $|mH|$, $m\gg 0$ and this is due to Proposition~\ref{bigzposcc} and Corollary~3.18 in \cite{Eno}.
\end{example}

\begin{proposition} \label{higherccconn}
Let $X$ be a normal projective variety over an algebraically closed field $k$.
Let $D$ be a chain-connected divisor on $X$.
Then $H^{0}(\O_D)\cong k$ holds.
\end{proposition}

\begin{proof}
We show the claim by induction on $n=\dim X$. The $n=2$ case is due to Lemma~\ref{ccconn}.
We take a general hyperplane $Y\in |mH|$ and consider the exact sequence
$$
0\to \O_D(-Y|_{D})\to \O_D\to \O_{Y\cap D}\to 0.
$$
Since $H^0(\O_D(-Y|_{D}))=0$, we have an injection $H^0(\O_D)\hookrightarrow H^0(\O_{Y\cap D})$.
Since $D$ is chain-connected with respect to $H$, the restriction $D|_{Y}$ on $Y$ is also chain-connected with respect to $H|_Y$.
Then we have $H^0(\O_{Y\cap D})\cong k$ by the inductive hypothesis, whence $H^0(\O_D)\cong k$ holds.
\end{proof}

Similarly, one can show the following:

\begin{proposition}
Let $X$ be a normal projective geometrically connected variety over a field $k$.
Let $D=\ulcorner M \urcorner$ be an effective divisor as in Example~\ref{cconnex}~$(\mathrm{ii})$.
Then $H^0(\O_D)/k$ is a purely inseparable field extension.
\end{proposition}

\begin{proof}
Note that the condition of divisors in Example~\ref{cconnex}~(ii) is preserved by any separable base change.
Then we may assume that $k$ is infinite.
By the hyperplane cutting argument as in the proof of Proposition~\ref{higherccconn}, we conclude that $H^0(\O_D)$ is a field.
Then the claim follows from the same argument in the proof of Corollary~\ref{bigzposinsep}.
\end{proof}

\section{Vanishing theorem on $H^1$}
\label{sec:Vanishing theorem on $H^1$}

In this section, we prove some vanishing theorems of Ramanujam, Kawamata--Viehweg and Miyaoka type for normal complete surfaces and normal projective varieties.

\subsection{Picard schemes and $\alpha(X,D)$}
We start with an elementary lemma, which is well known to experts:
\begin{lemma} \label{Frobinj}
Let $X$ be a proper scheme over an algebraically closed field $k$ of characteristic $p>0$ with $H^{0}(\O_{X})\cong k$.
Then the following are equivalent.

\smallskip

\noindent
$(\mathrm{i})$ The Frobenius map on $H^{1}(\O_X)$ is injective.

\smallskip

\noindent
$(\mathrm{ii})$ $H^1(X_{\mathrm{fppf}},\alpha_{p,X})=0$, that is, all $\alpha_p$-torsors over $X$ are trivial.

\smallskip

\noindent
$(\mathrm{iii})$ The Picard scheme $\Pic_{X/k}$ does not contain $\alpha_{p}$.

\smallskip

\noindent
$(\mathrm{iv})$ Any infinitesimal subgroup scheme of $\Pic^{\circ}_{X/k}$ is linearly reductive, that is, of the form $\prod \mu_{p^{n_i}}$.
\end{lemma}

\begin{proof}
The equivalence of (i) and (ii) follows from the fact $H^1(X_{\mathrm{fppf}},\alpha_{p,X})\cong \{\eta \in H^1(\O_X)\ |\ F(\eta)=0\}$, which is obtained by the exact sequence $0\to \alpha_p\to \G_a\xrightarrow{F} \G_a^{(p)}\to 0$
of commutative group schemes over $k$, where $F$ is the relative Frobenius map.
The equivalence of (ii) and (iii) follows from Proposition~(6.2.1) in \cite{Ray} with $T=\Spec k$ and $M=\alpha_p$.
Indeed, there is a natural isomorphism $H^1(X_{\mathrm{fppf}},\alpha_{p,X})\cong \Hom_{\mathrm{GSch}/k}(\alpha_p,\Pic_{X/k})$ as abelian groups.
For the equivalence of (iii) and (iv), it suffices to show that for any $m\ge 1$, the $m$-th Frobenius kernel $\Pic_{X/k}[F^m]$ does not contain $\alpha_p$ if and only if it is linearly reductive.
This holds true for any infinitesimal group scheme by Lemma~2.3 in \cite{LMM}.
\end{proof}

\begin{definition}
Let $X$ and $D$ be proper schemes over a field $k$ and $\tau \colon D\to X$ a morphism.
Then we denote by $\alpha(X,D)$ the kernel of $\tau^{*}\colon H^1(\O_X)\to H^{1}(\O_D)$.
It can be identified with the Lie algebra of the kernel of the homomorphism $\tau^{*}\colon \Pic_{X/k}\to \Pic_{D/k}$ of Picard schemes defined by the pull-back of line bundles (cf.\ {\cite[p.231, Theorem~1]{BLR}}).
\end{definition}

Let $\phi\colon E\to D$ and $\tau\colon D\to X$ be two morphisms of proper schemes over a field $k$. 
In this subsection, we are going to study the relations of $\alpha(X,D)$ and $\alpha(X,E)$, which will be used in Section~5.
Let $\phi^{*}\colon \Pic^{\circ}_{D/k}\to \Pic^{\circ}_{E/k}$ and $\tau^{*}\colon \Pic^{\circ}_{X/k}\to \Pic^{\circ}_{D/k}$ be the corresponding homomorphisms of Picard schemes.

\begin{lemma} \label{alphacoincide}
Let $\phi\colon E\to D$ and $\tau\colon D\to X$ be morphisms of proper schemes over a field $k$ and assume one of the following:

\smallskip

\noindent
$(\mathrm{i})$ $\Char k=0$, $X$ is normal and $\Ker(\phi^{*})$ is affine, or

\smallskip

\noindent
$(\mathrm{ii})$ $\Char k>0$, $H^0(\O_X)\cong k$, the Frobenius map on $H^1(\O_X)$ is injective and $\Ker(\phi^{*})$ is unipotent.

\smallskip

\noindent
Then $\alpha(X,D)=\alpha(X,E)$ holds.
\end{lemma}

\begin{proof}
It suffices to show that $\Ker(\tau^{*})^{\circ}=\Ker(\phi^{*}\circ \tau^{*})^{\circ}$.
Taking the base change to an algebraic closure $\overline{k}$ of $k$, we may assume that $k$ is algebraically closed.
Let us write $Q:=\Image(\tau^{*}|_{\Ker(\phi^{*}\circ \tau^{*})^{\circ}})$ and consider the exact sequence
$$
1\to \Ker(\tau^{*})^{\circ}\to \Ker(\phi^{*}\circ \tau^{*})^{\circ} \to Q \to 1.
$$

First we assume the condition (i).
Then $\Pic^{\circ}_{X/k}$ is proper over $k$ since $X$ is normal ({\cite[expos\'{e} 236, Th.\ 2.1 (ii)]{FGA}}).
Thus $\Ker(\phi^{*}\circ \tau^{*})^{\circ}$ and  $Q:=\Image(\tau^{*}|_{\Ker(\phi^{*}\circ \tau^{*})^{\circ}})$ are also proper over $k$.
On the other hand, since $\Ker(\phi^{*})$ is affine, so is $Q$.
Thus we conclude that $Q$ is infinitesimal, whence $Q$ is trivial due to Cartier's theorem.

We assume the condition (ii).
It suffices to show that $Q[F]$ is trivial.
By Lemma~\ref{Frobinj}, the group scheme $\Pic_{X/k}[F]$ is linearly reductive.
Hence $\Ker(\phi^{*}\circ \tau^{*})[F]$ and $Q[F]$ are also linearly reductive since the linear reductivity is preserved by taking subgroup schemes and quotient group schemes.
On the other hand, $Q[F]$ is unipotent since so is $\Ker(\phi^{*})$. 
Thus $Q[F]$ is trivial.
\end{proof}

Let us recall the structure of the generalized Jacobian $\Pic^{\circ}_{C/k}$ of a proper curve $C$.
The following combinatrial data is useful to counting the rank of maximal tori in $\Pic^{\circ}_{C/k}$:

\begin{definition}
Let $C$ be a proper curve (i.e., purely $1$-dimensional proper scheme) over an algebraically closed field $k$.
Then the {\em extended dual graph} $\Gamma(C)$ of $C$ is defined as follows:
The vertex set of $\Gamma(C)$ is the union of the integral subcurves $\{C_i\}_{i}$ in $C$ and the singularities $\{x_{\lambda}\}_{\lambda}$ of the reduced scheme $C_{\red}$ of $C$.
For each singular point $x_{\lambda}$ of $C_{\red}$, we denote by $B_1,\ldots, B_m$ all the local analytic branches of $C_{\red}$ (that is, the minimal primes of the complete local ring $\widehat{\O}_{C_{\red},x_{\lambda}}$).
For each $B_{j}$, let $C_{i(j)}$ denote the corresponding integral curve in $C$.
Then the edges of $\Gamma(C)$ are defined by connecting $x_{\lambda}$ and $C_{i(j)}$ for each branch $B_j$.

For a proper curve $C$ over an arbitrary field $k$, we define the extended dual graph of $C$ by that of $C_{\overline{k}}=C\times_{k} \overline{k}$, where $\overline{k}$  is an algebraic closure of $k$.
\end{definition}

\begin{proposition}[{\cite[Section~9.2, Proposition~10]{BLR}}] \label{maxtorus}
Let $C$ be a proper curve over an algebraically closed field $k$.
Then the rank of maximal tori of $\Pic^{\circ}_{C/k}$ is the first Betti number $b_{1}(\Gamma(C))$ of the extended dual graph of $C$.
\end{proposition}

\begin{proof}
For the convenience of readers, we sketch the proof.
We may assume that $C$ is connected. 
Let $\widetilde{C}$ denote the normalization of the reduced scheme $C_{\red}$ of $C$.
Then the natural map $\pi \colon \widetilde{C}\to C$ can be decomposed into $\widetilde{C}\to C'\to C_{\red}\to C$ as in the argument in {\cite[Section~9.2]{BLR}}, where the intermediate curve $C'$ is canonically determined as the highest birational model of $C_{\red}$ which is homeomorphic to $C_{\red}$.
Since $\Gamma(C)=\Gamma(C_{\red})=\Gamma(C')$ and the kernel of $\Pic^{\circ}_{C/k}\to \Pic^{\circ}_{C'/k}$ is unipotent by Propositions~5 and 9 in {\cite[Section~9.2]{BLR}},
we may assume $C=C'$.
Let $x_{\lambda}$, $\lambda=1,\ldots,N$ denote the singular points of $C$ and for each $\lambda$, let $\widetilde{x}_{\lambda,\mu}$, $\mu=1,\ldots,n_\lambda$ denote the points of $\widetilde{C}$ lying over $x_{\lambda}$.
Let $C_i$, $i=1,\ldots, r$ denote the integral components of $C$.
Taking the cohomology of the exact sequence $1\to \O^{*}_C\to \pi_{*}\O^{*}_{\widetilde{C}}\to \mathcal{T}\to 1$,
we obtain a long exact sequence
$$
1\to k^{*}\to \prod_{i=1}^{r} k^{*}\to \prod _{\lambda=1}^{N}(\prod_{\mu=1}^{n_{\lambda}}k(\widetilde{x}_{\lambda,\mu})^{*})/k(x_{\lambda})^{*}\to \Pic(C)\to \Pic(\widetilde{C})\to 1,
$$ 
where we note that the cokernel $\mathcal{T}$ is a torsion sheaf supported at the singular points $x_{\lambda}$.
Then the kernel of $\pi^{*}\colon \Pic_{C/k}\to \Pic_{\widetilde{C}/k}$ is a torus of rank $\sum_{\lambda=1}^{N}(n_{\lambda}-1)-r+1$.
On the other hand, the number of vertices and edges of $\Gamma(C)$ are $r+N$ and $\sum_{\lambda=1}^{N}n_{\lambda}$, respectively.
Thus the topological Euler number of the graph is $\chi_{\mathrm{top}}(\Gamma(C)))=r+N-\sum_{\lambda=1}^{N}n_{\lambda}$ and then the first Betti number is 
$$
b_1(\Gamma(C))=1-\chi_{\mathrm{top}}(\Gamma(C)))=\sum_{\lambda=1}^{N}(n_{\lambda}-1)-r+1,
$$
which completes the proof.
\end{proof}

\begin{lemma}\label{kernelunip}
Let $\phi\colon E\to D$ be a morphism between proper schemes over a field $k$.
Then the kernel of $\phi^{*}\colon \Pic_{D/k}^{\circ}\to \Pic_{E/k}^{\circ}$ is unipotent if one of the following holds.

\smallskip

\noindent
$(\mathrm{i})$ The canonical immersion $D_{\red}\hookrightarrow D$ factors through $\phi$, or

\smallskip

\noindent
$(\mathrm{ii})$ $D$ is a curve and there exists a birational morphism $\widehat{D}\to D_{\red}$ with $b_1(\Gamma(D))=b_1(\Gamma(\widehat{D}))$ 
such that the composition $\widehat{D}\to D_{\red}\hookrightarrow D$ factors through $\phi$.
\end{lemma}

\begin{proof}
In order to show the assertion for (i), we may assume that $E=D_{\red}$.
By taking a filtration of first order thickenings $D_{\red}\hookrightarrow D_1\hookrightarrow \cdots \hookrightarrow D_N=D$, 
we further assume that $E\hookrightarrow D$ is a first order thickening.
Then one can show the assertion easily by taking the cohomology of the exact sequence
$$
0\to \mathcal{I}_{E/D}\to \O_{D}^{*}\to \O_{E}^{*}\to 1,
$$
where the map on the left sends a local section $a$ to $1+a$. 
For the case (ii), we may assume $E=\widehat{D}$.
By taking the base change to an algebraic closure $\overline{k}$ of $k$, we may assume that $k$ is algebraically closed.
Note that the kernel of $\phi^{*}\colon \Pic_{D/k}^{\circ}\to \Pic_{\widehat{D}/k}^{\circ}$ is affine, and it is unipotent if and only if it does not contain a torus (cf.\ Corollaries~11 and 12 in {\cite[Section~9.2]{BLR}}).
Then the claim follows from Proposition~\ref{maxtorus} and the surjectivity of $\phi^{*}\colon \Pic_{D/k}^{\circ}\to \Pic_{\widehat{D}/k}^{\circ}$.
\end{proof}

By combining Lemma~\ref{alphacoincide} with Lemma~\ref{kernelunip}, we obtain the following:

\begin{proposition}\label{alphacoinprop}
Let $\phi\colon E\to D$ and $\tau\colon D\to X$ be morphisms between proper schemes over a field $k$.
If $\Char k>0$ $($resp.\ $\Char k=0$$)$, we further assume that $H^0(\O_X)\cong k$ and the Frobenius map on $H^1(\O_X)$ is injective $($resp.\ $X$ is normal$)$.
Then $\alpha(X,D)=\alpha(X,E)$ holds if one of the following holds:

\smallskip

\noindent
$(\mathrm{i})$ The canonical immersion $D_{\red}\hookrightarrow D$ factors through $\phi$.

\smallskip

\noindent
$(\mathrm{ii})$ $D$ is a curve and there exists a birational morphism $\widehat{D}\to D_{\red}$
such that the composition $\widehat{D}\to D_{\red}\hookrightarrow D$ factors through $\phi$.
Moreover, we further assume $b_1(\Gamma(D))=b_1(\Gamma(\widehat{D}))$ if $\Char k>0$.
\end{proposition}

\begin{remark}
(1) 
Proposition~\ref{alphacoinprop} also holds when $\phi$ and $\tau$ are morphisms between compact complex analytic spaces and $X$ is a normal compact analytic variety in Fujiki's class $\mathcal{C}$ because the Picard variety $\Pic^{\circ}(X)\cong H^1(X, \O_X)/H^1(X,\Z)_{\mathrm{free}}$ is a compact torus.

\smallskip

\noindent
(2)
If $\Char k>0$ and the condition (i) in Proposition~\ref{alphacoinprop} holds, one can also show $\alpha(X,D)=\alpha(X,E)$ without using Picard schemes by the same proof of Lemma~6 in \cite{Ram}.
\end{remark}

The following is a generalization of Lemma~2.3 in \cite{Fra}.

\begin{corollary} \label{Francia}
Let $X$ be a normal proper variety over a field $k$, or normal compact analytic variety in Fujiki's class $\mathcal{C}$.
Let $D_1$ and $D_2$ be closed subschemes of $X$.
Let $\pi\colon X\to Y$ be a birational morphism to a normal variety $Y$.
We assume the following three conditions:

\smallskip

\noindent
$(\mathrm{i})$
$D_1$ and $D_2$ are not contained in the exceptional locus of $\pi$ and $H^1(\O_Y)\cong H^1(\O_{X})$,

\smallskip

\noindent
$(\mathrm{ii})$
The reduced image of $D_1$ by $\pi$ coincides with that of $D_2$ and it is a curve, which is denoted by $D$, and

\smallskip

\noindent
$(\mathrm{iii})$
If $\Char k>0$, the Frobenius map on $H^1(\O_X)$ is injective and $b_1(\Gamma(D))=b_1(\Gamma(\widehat{D}))$, where $\widehat{D}$ is the proper transform of $D$ on $X$.

\smallskip

\noindent
Then we have $\alpha(X,D_1)=\alpha(X,D_2)\cong \alpha(Y,D)$.
\end{corollary}

\begin{proof}
We may assume that $H^0(\O_X)\cong k$ by replacing $k$ to the field $H^0(\O_X)$.
Since $\pi(\widehat{D})=D$, we may assume $D_1=\widehat{D}$.
In particular, there is a closed immersion $D_1\hookrightarrow D_2$.
From Proposition~\ref{alphacoinprop}, we have $\alpha(Y,D)=\alpha(Y,D_2)=\alpha(Y,D_1)$.
On the other hand, it follows from $H^1(\O_Y)\cong H^1(\O_{X})$ that $\alpha(X,D_i)\cong \alpha(Y,D_i)$, $i=1,2$.
Hence we conclude $\alpha(X,D_1)=\alpha(X,D_2)\cong \alpha(Y,D)$.
\end{proof}

\subsection{Vanishing of $\alpha(X,D)$}

In this subsection, we are going to study the vanishing of $\alpha(X,D)$ when $D$ is a divisor on a variety $X$.

\begin{definition} \label{pencil}
Let $X$ be a normal proper variety over a field $k$ and $D$ a divisor on $X$.
The complete linear system $|D|$ on $X$ defines a rational map $\phi_{|D|}\colon X\dasharrow \PP^N$, $N=\dim |D|$. 
Taking a resolution $X'\to X$ of the indeterminacy of $\phi_{|D|}$ and the Stein factorization, we obtain the morphisms $X\leftarrow X'\to B\to \PP^N$, where the middle map is a fiber space.
Then we say that $|D|$ is {\em composed with a $($resp.\ rational, irrational$)$ pencil} if $\dim B=1$ (resp.\ and $H^1(\O_B)=0$, $H^1(\O_B)\neq 0$).
\end{definition}

\begin{proposition}\label{surfalphavan} 
Let $X$ be a normal proper surface over a field $k$ or analytic Moishezon surface.
Let $D$ be an effective and big divisor on $X$.
Then $\alpha(X,D)=0$ $($resp. $\alpha(X,D)_s=0$$)$ holds if $\Char k=0$ $($resp.\ $\Char k>0$ and either $k=\overline{k}$ or $H^1(\O_X)_n=0$$)$.
\end{proposition}

\begin{proof}
We follow Mumford's argument in \cite[p.\ 99]{Mum2}.
Since $H^0(\O_X)$ is a field, we may assume $H^0(\O_X)=k$.
By taking the base change to a separable closure of $k$,
we may also assume that $k$ is separably closed.
Let $\tau\colon D\hookrightarrow X$ denote the natural immersion.

First we suppose that $\alpha(X,D)\neq 0$, i.e., $\Ker(\tau^{*})^{\circ}\neq 1$ and that  $H^1(\O_X)_n=0$ when $\Char k>0$.
Then we can take a subgroup scheme $\mu_{p}\subset \Ker(\tau^{*})^{\circ}$, where $p$ is a prime number and $p=\Char k$ when $\Char k>0$.
Indeed, the characteristic $0$ case is trivial and so we may assume that $\Char k=p>0$.
Then by Lemma~\ref{Frobinj}, $\Ker(\tau^{*})[F]\times_{k}\overline{k}$ is isomorphic to the product $\prod_{i}\mu_{p^{n_i}}$.
Since $k$ is separably closed, $\Ker(\tau^{*})[F]$ is also isomorphic to $\prod_{i}\mu_{p^{n_i}}$. 
In particular, $\Ker(\tau^{*})^{\circ}$ contains at least one $\mu_{p}$.
Thus by the natural isomorphism $H^1(X_{\mathrm{et}},(\Z/p \Z)_{X})\cong \Hom_{\mathrm{GSch}/k}(\mu_{p},\Pic_{X/k})$, we can take a non-trivial \'{e}tale cyclic covering $\pi\colon Y\to X$ of degree $p$ with $\pi^{*}D=\sum_{i=1}^{p}D_{i}$, where all $D_{i}$ are disjoint and $D_{i}\cong D$.

If $\alpha(X,D)_s\neq 0$ and the base field $k$ is algebraically closed
 of characteristic $p>0$, then there exists a non-zero element $\eta\in \alpha(X,D)$ such that $F(\eta)=\eta$.
Thus we can also take a non-trivial \'{e}tale cyclic covering $\pi\colon Y\to X$ of degree $p$ with $\pi^{*}D=\sum_{i=1}^{p}D_{i}$ corresponding to $\eta$ since $H^1(X_{\mathrm{et}},(\Z/p \Z)_{X})\cong \{\eta \in H^1(\O_X)\ |\ F(\eta)=\eta\}$.

Let $D=P+N$ and $D_i=P_i+N_i$ respectively be the Zariski decompositions of $D$ and $D_i$.
Then $\pi^{*}P=\sum_{i=1}^{p}P_{i}$ holds and it is nef and big, which contradicts Lemma~\ref{nefbignumconn}.
\end{proof}

\begin{proposition}\label{surfpencilvan} 
Let $X$ be a normal proper surface over a field $k$ of characteristic $0$ or analytic in Fujiki's class $\mathcal{C}$.
Let $D$ be an effective divisor on $X$ with $\kappa(D)=1$.
Then $\alpha(X,D)=0$ holds if $|mD|$ is composed with a rational pencil for some $m>0$.  
\end{proposition}

\begin{proof}
The proof is identical to that of (12.8) Lemma in \cite{BHPV}.    
For the readers' convenience, we give a proof.
Since $\alpha(X,D)=\alpha(X,mD)\subset \alpha(mD-\mathrm{Fix}|mD|)$ holds for any $m\ge 1$ from Proposition~\ref{alphacoinprop},
we may assume that $|D|$ has no fixed parts.
Then it is base point free since $\kappa(D)=1$.
Let $f\colon X\to B$ be the fibration with connected fibers induced by $|D|$.
Then we can write $D=\sum_{i=1}^{l}F_{t_i}$, where $F_{t_i}=f^{-1}(t_i)$ are the fibers of $f$ at some closed points $t_i\in B$.
By the Leray spectral sequence $H^p(R^{q}f_{*}\O_X)\Rightarrow H^{p+q}(\O_X)$ and the assumption $H^1(\O_B)=0$, 
we have $H^1(\O_X)\cong H^0(R^{1}f_{*}\O_X)$.
Let $\eta$ be an element of $\alpha(X,D)$.
Then $\eta|_{F_t}=0$ for some $t\in B$ (for example, take $t=t_1$), that is, $\eta\in \alpha(X,F_t)$.
By Proposition~\ref{alphacoinprop}, we have $\eta\in \alpha(X,nF_t)$ for any $n\ge 1$.
Thus the formal function theorem implies that $\eta$ maps to $0$ by the composition $H^1(\O_X)\cong H^0(R^{1}f_{*}\O_X)\to (R^{1}f_{*}\O_X)_{t}$.
Hence there exists an open neighborhood $U\subset B$ of $t$ such that $\eta|_{U}=0$, which implies $\eta=0$ since $R^{1}f_{*}\O_X$ is locally free (note that $f$ contains no wild fibers by the assumption $\Char k=0$).
\end{proof}

\begin{example}
When $\Char k=p>0$, there exist counter-examples to Proposition~\ref{surfpencilvan} as follows:
Let $G=\Z/p \Z$ be a constant group scheme over an algebraically closed field $k$ with $\Char k=p>0$ and $g\in G$ a generator.
Then $G$ acts on $\mathbb{
A}^1$ as the translation $g\colon t\mapsto t+1$.
This extends to an action on $\PP^1$.
Let $E$ be an ordinary elliptic curve and take a $p$-torsion point $a\in E(k)$.
Then $G$ acts freely on $E$ as $g\colon x\mapsto x+a$.
Thus the diagonal action of $G$ to $E\times \PP^1$ is free and the quotient $X:=(E\times \PP^1)/G$ admits a structure of elliptic surfaces $f\colon X\to \PP^1/G\cong \PP^1$ via the second projection.
This admits one wild fiber $f^{-1}(\infty)=pE_{\infty}$ at the infinity point $\infty\in \PP^1$.
Then a simple calculation shows that $R^1f_{*}\O_X\cong \O_{\PP^1}(-1)\oplus \mathcal{T}$ holds, where $\mathcal{T}$ is a torsion sheaf supported at $\infty$ with length $1$ (cf.\ {\cite[p.313, Section~8]{KaUe}}).  
Now we consider a fiber $D:=f^{-1}(t)$ at a point $t\neq \infty$.
Since $H^1(\O_X)\cong H^0(R^1f_{*}\O_X)\cong H^0(\mathcal{T})$, we have $\alpha(X,D)=H^1(\O_X)\cong k$.
We note that the Frobenius map on $H^1(\O_X)$ is injective since $H^1(\O_X)\cong H^1(\O_{E_{\infty}})$ and $E_{\infty}$ is ordinary.
\end{example}

Next we consider the higher dimensional cases.
The following is a generalization of Theorem~2.1 in \cite{AlTo}.
\begin{proposition} \label{alphavan}
Let $X$ be a normal projective variety of dimension $n\ge 2$ over an infinite field $k$.
Let $D$ be an effective divisor on $X$ such that 
$D|_{S}$ is big on a complete intersection surface $S:=H_1\cap \cdots \cap H_{n-2}$ for general hyperplanes $H_1,\ldots,H_{n-2}$ on $X$.
Then $\alpha(X,D)=0$ $($resp.\ $\alpha(X,D)_s=0$$)$ holds if $\Char k=0$ $($resp.\ $\Char k>0$ and either $k=\overline{k}$ or $H^1(\O_X)_n=0$$)$.
\end{proposition}

\begin{proof}
The proof is similar to that of Proposition~\ref{surfalphavan}.
Suppose that $\alpha(X,D)\neq 0$ (resp.\ $\alpha(X,D)_s\neq 0$).
Then we can also take a non-trivial \'{e}tale cyclic covering $\pi\colon Y\to X$ of prime degree $p$ with $\pi^{*}D=\sum_{i=1}^{p}D_{i}$, where all $D_{i}$ are disjoint and $D_{i}\cong D$.
Let $H_1,\ldots, H_{n-2}$ be general hyperplanes on $X$ such that $S:=H_1\cap \cdots \cap H_{n-2}$ is a normal surface.
Then $S':=\pi^{-1}(S)$ is normal and $(\pi|_{S'})^{*}(D|_{S})=D_{1}|_{S'}+\cdots+D_{p}|_{S'}$ is  big, which is a contradiction.
\end{proof}

\begin{proposition} \label{pencilvan}
Let $X$ be a normal projective variety of dimension $n\ge 2$ over a field $k$ of characteristic $0$.
Let $D$ be an effective divisor on $X$ such that 
$|mD|$ is composed with a rational pencil for some $m>0$.
Then $\alpha(X,D)=0$ holds.
\end{proposition}

\begin{proof}
We use induction on the dimension $n=\dim X$.
The $n=2$ case is due to Propositions~\ref{surfalphavan} and \ref{surfpencilvan}.
We assume $n\ge 3$.
Let us take a general hyperplane $Y$ on $X$.
Then $D|_{Y}$ satisfies $\kappa(D|_{Y})\ge 2$ or $|mD|$ is composed with a rational pencil.
Hence the claim holds by the natural inclusion $\alpha(X,D)\hookrightarrow \alpha(Y,D|_Y)$ due to Enriques--Severi--Zariski's lemma, Proposition~\ref{alphavan} and the inductive assumption.
\end{proof}

\subsection{Vanishing on $H^1$}

Combining Propositions~\ref{surfalphavan} and \ref{surfpencilvan} with Lemma~\ref{ccconn}, we obtain the following vanishing theorem on normal surfaces:

\begin{theorem} \label{ccvan}
Let $X$ be a normal proper surface over a field $k$ or analytic in Fujiki's class $\mathcal{C}$.
Let $D$ be a chain-connected divisor on $X$ 
having a prime component $C$ with $H^{0}(\O_X)\cong H^{0}(\O_C)$.
Assume that $|mD|$ is not composed with an irrational pencil and has positive dimension for some $m>0$.
If $\Char k>0$, we further assume that $D$ is big and the Frobenius map on $H^1(\O_X)$ is injective.
Then we have $H^1(\O_{X}(-D))=0$ or equivalently, $H^1(\O_{X}(K_X+D))=0$.
\end{theorem}

The following is a positive characteristic analog of Theorem~4.1 in \cite{Eno} and includes a Kawamata--Viehweg type vanishing theorem for surfaces in positive characteristic:
\begin{theorem} \label{bigzposvan}
Let $X$ be a normal proper geometrically connected surface over a perfect field $k$ of positive characteristic.
Let $D$ be a big divisor on $X$ and $D=P_{\Z}+N_{\Z}$ the integral Zariski decomposition as in Theorem~3.5 in \cite{Eno}.
Let $\mathcal{L}_D$ and $\mathcal{L'}_D$ respectively be the rank $1$ sheaves on $N_{\Z}$ defined by the cokernel of the homomorphisms $\O_X(K_X+P_{\Z})\to \O_X(K_X+D)$ and $\O_X(-D)\to \O_X(-P_{\Z})$ induced by multiplying a defining section of $N_{\Z}$.
If $\dim|D|\ge \dim H^{1}(\O_X)_n$, then we have 
$$
H^1(X,\O_X(K_X+D))\cong H^1(N_{\Z},\mathcal{L}_D)
$$ 
and
$$
H^1(X,\O_X(-D))\cong H^0(N_{\Z},\mathcal{L'}_D).
$$
\end{theorem}

\begin{proof}
Since $\dim |D|=\dim |P_{\Z}|$, it suffices to show the vanishing of $H^1(X,\O_X(-D))$ under the additional assumption that $D$ is $\Z$-positive (cf.\ the proof of Theorem~4.1 in \cite{Eno}). 
We may assume that $k$ is algebraically closed by Lemma~\ref{zposbasechange}.
Now we use a slight modification of Fujita's argument (cf.\ (7.4)~Theorem in \cite{Fuj}).
First note that $H^0(\O_{D_s})\cong k$ holds for any member $D_s\in |D|$ from Corollary~\ref{bigzposconn}.
Thus each non-zero section $s\in H^0(\O_X(D))$ defines an injection $\times s\colon H^1(\O_X(-D))\hookrightarrow H^1(\O_X)$.
Moreover, the image of $\times s$ is contained in $H^1(\O_X)_n$ by Proposition~\ref{surfalphavan}.
Let $U:=H^0(\O_X(D))$, $V:=H^1(\O_X(-D))$, $W:=H^1(\O_X)_n$ and $M:=\Hom_{k}(V,W)$ for simplicity and consider these as affine varieties over $k$.
Then the correspondence $s\mapsto \times s$ as above defines a $k$-morphism $\Phi\colon U\to M$.
Now we suppose that $V\neq 0$.
Note that $1 \le \dim V\le \dim W<\dim U$ holds by assumption and the above argument.
Then $\Phi$ induces a morphism $\overline{\Phi}\colon \PP (U^{*})=|D|\to \Gr(r,W)$, where $\Gr(r,W)$ is the Grassmann variety parametrizing all $r:=\dim V$-dimensional $k$-linear subspaces of $W$.
Since $\dim W<\dim U$, it follows from Corollary~3.2 in \cite{Tan} that the morphism $\overline{\Phi}$ is constant. 
We denote by $(I\subset W)$ the image of $\overline{\Phi}$.
Thus $\Phi$ induces the morphism 
$$
\det \Phi\colon U\xrightarrow{\Phi} \Hom_k(V,I)\xrightarrow{\det} \Hom_k(\wedge^{r}V,\wedge^{r}I)\cong \Af^1_{k},
$$
the restriction of which to $U\setminus \{0\}$ is non-zero everywhere.
This contradicts $\dim U\ge 2$ since $(\det \Phi)^{-1}(0)$ must be a divisor.
\end{proof}

\begin{corollary}[Kawamata--Viehweg type vanishing theorem]\label{KVvan}
Let $X$ be a normal proper surface over a perfect field $k$ of positive characteristic with $H^0(\O_X)\cong k$.
Let $M$ be a nef and big $\R$-divisor on $X$.
If $\dim|\ulcorner M\urcorner |\ge \dim H^1(\O_X)_{n}$ holds, then we have
$H^i(X,\O_X(K_X+\ulcorner M\urcorner))=0$ for any $i>0$.
\end{corollary}

For the higher dimensional cases, by combining Propositions~\ref{alphavan} and \ref{pencilvan} with Proposition~\ref{higherccconn}, we obtain the following:

\begin{theorem}[Generalized Ramanujam vanishing theorem] \label{RamvanthmI}
Let $X$ be a normal projective variety of dimension $n\ge 2$ over an algebraically closed field $k$.
If $\Char k>0$, we assume that the Frobenius map on $H^1(\O_X)$ is injective.
Let $D$ be a chain-connected divisor on $X$ which satisfies one of the following conditions:

\smallskip

\noindent
$(\mathrm{i})$ $D|_{S}$ is big on a complete intersection surface $S:=H_1\cap \cdots \cap H_{n-2}$ for general hyperplanes $H_1,\ldots,H_{n-2}$, or

\smallskip

\noindent
$(\mathrm{ii})$ $|mD|$ is composed with a rational pencil for some $m>0$ and $\Char k=0$.

\smallskip

\noindent
Then $H^1(X,\O_X(-D))=0$ holds.
\end{theorem}

The following can be seen as a higher dimensional generalization of Corollary~\ref{KVvan} and Theorem~(2.7) in \cite{Miy}.

\begin{theorem}[Generalized Miyaoka vanishing theorem] \label{Miyvanthm}
Let $X$ be a normal projective geometrically connected variety of dimension $n\ge 2$ over an infinite perfect field $k$.
Let $D$ be a divisor on $X$.
We assume the following three conditions:

\smallskip

\noindent
$(\mathrm{i})$ $D=\ulcorner M \urcorner +E$ for some $\R$-divisor $M$ and the sum of prime divisors $E=\sum_{i=1}^{m}E_i$ $($possibly $m=0$$)$.

\smallskip

\noindent
$(\mathrm{ii})$ There exist $n-2$ hyperplanes 
 $H_1,\ldots,H_{n-2}$ on $X$ with $S:=H_1\cap \cdots \cap H_{n-2}$ a normal surface such that 
$M|_{S}$ is nef, $D|_{S}$ is big and for each $j$,
$$
H_1\cdots H_{n-2}(\ulcorner M\urcorner+\sum_{i=1}^{j-1}E_i) E_j>0.
$$

\smallskip

\noindent
$(\mathrm{iii})$ $\Char k=0$ or $\dim|D|\ge \dim H^1(\O_X)_n$.

\smallskip

\noindent
Then $H^1(X,\O_X(-D))=0$ holds.
\end{theorem}

\begin{proof}
We may assume that $k$ is algebraically closed since the conditions (i), (ii) and (iii) are preserved by any separable base change.
We note that for any effective divisor $D$ satisfying the conditions (i) and (ii), $D|_{S}$ is big $\Z$-positive on a general complete intersection surface $S=H_1\cap \cdots \cap H_{n-2}$.
Indeed, this can be checked from the fact that $D+C$ is $\Z$-positive for any $\Z$-positive divisor $D$ and any prime divisor $C$ on $S$ with $DC>0$.
Thus $H^0(\O_D)\cong k$ holds from Corollary~\ref{bigzposinsep}.

First assume that $\Char k=0$.
We use induction on $n=\dim X$.
The $n=2$ case is due to Theorem~4.1 (1) in \cite{Eno}.
Assume that $n\ge 3$.
We take a general hyperplane $Y$ on $X$.
Then $D|_Y=\ulcorner M |_Y\urcorner+E|_Y$ also satisfies the conditions (i), (ii) and (iii) in Theorem~\ref{Miyvanthm}.
Then the claim follows from the injection $H^1(\O_X(-D))\hookrightarrow H^1(\O_Y(-D|_Y))$ due to Enriques--Severi--Zariski's lemma and the inductive assumption.

We assume that $\Char k>0$.
Note that $\alpha(X,D)$ is contained in $H^1(\O_X)_n$ by Proposition~\ref{alphavan}.
Hence the proof is identical to that of Theorem~\ref{bigzposvan}.
\end{proof}

\begin{remark}
(1) Theorems~\ref{ccvan} and \ref{RamvanthmI} can be seen as generalizations of Ramanujam's $1$-connected vanishing for smooth surfaces (cf.\ {\cite[Chapter~IV, (12.5)~Theorem]{BHPV}}).
Theorem~\ref{RamvanthmI} is also a generalization of Theorem~2 in \cite{Ram}.

\smallskip

\noindent
(2) Theorem~\ref{Miyvanthm} recovers Theorem~3.5.3 in \cite{Fuj}.
Indeed, for a smooth complete variety $X$ of dimension $\ge 2$ and a nef $\R$-divisor $M$ with $\nu(M)\ge 2$, we reduce the vanishing $H^1(X,\O_X(-\ulcorner M \urcorner))=0$ to Theorem~\ref{Miyvanthm} as follows:
Take a birational morphism $\pi\colon X'\to X$ from a smooth projective variety by using Chow's lemma and apply the Leray spectral sequence $H^p(R^{q}\pi_{*}\O_{X'}(-\ulcorner \pi^{*}M \urcorner))\Rightarrow H^{p+q}(\O_X(-\ulcorner M \urcorner))$.

\smallskip

\noindent
(3) The conditions (i) and (ii) in Theorem~\ref{Miyvanthm} are satisfied for any divisor $D$ of the form $D=\ulcorner M \urcorner$, $M$ is an $\R$-divisor which is nef in codimension $1$ and satisfies $\nu(M)\ge 2$ or $\kappa(D)\ge 2$.
\end{remark}

\begin{example}
(1) Raynaud surfaces (\cite{Ray2}, \cite{Muk}) are smooth projective surfaces $X$ of positive characteristic with ample divisors $D$ with $H^1(X,\O_X(-D))\neq 0$.
By construction, we can take the divisor $D$ effective.
Thus these examples show that Theorem~\ref{bigzposvan} does not hold if we only assume the weaker condition that $|D|\neq \emptyset$.

\smallskip

\noindent
(2) The examples constructed in \cite{CaTa} (resp.\ in \cite{Ber}) are smooth (resp.\ klt) rational surfaces $X$ with a divisor $D$ of the form $D=\ulcorner M \urcorner$, a nef and big $\Q$-divisor $M$ (resp.\ an ample divisor $D$) on $X$ such that $H^1(X,\O_X(-D))\neq 0$.
Here by Theorem~\ref{bigzposvan}, we can not take the divisor $D$ effective because $H^1(\O_X)=0$ in this case. 

\smallskip

\noindent
(3) If the base field $k$ is not perfect, Theorem~\ref{bigzposvan} does not hold.
Indeed, Maddock \cite{Mad} constructed a regular del Pezzo surface $X_2$ over an imperfect field $k$ of characteristic $2$ with $\dim H^1(\O_{X_2})=1$ and $K_{X_2}^2=2$.
Then one sees by the Riemann--Roch theorem that $\dim |-K_{X_2}|\ge \dim H^1(\O_{X_2})=1$.

\smallskip

\noindent
(4) For any $i\ge 2$, there exist a normal projective variety $X$ of dimension $\ge 3$ and an ample Cartier divisor $D$ on $X$ such that $H^i(X,\O_{X}(-D))\neq 0$ even for characteristic $0$ (\cite{Som}).
Thus the similar results of Theorem~\ref{Miyvanthm} for the vanishing on $H^i$, $i\ge 2$ can not be expected. 
\end{example}

\section{Adjoint linear systems for effective divisors}
\label{sec:Adjoint linear systems for effective divisors}

We are now going to apply our vanishing theorems to the study of adjoint linear systems on normal surfaces. 
In this section, let $X$ be a normal proper surface over a field $k$ or a normal compact analytic surface in Fujiki's class $\mathcal{C}$.
First, let us recall the invariant $\delta_{\zeta}(\pi,Z)$ for the germ of a cluster $(X,\zeta)$. 

\begin{definition}[cf.\ Definition~5.1 in \cite{Eno}] \label{delta}
Let $\zeta$ be a {\em cluster} on $X$, that is, a $0$-dimensional subscheme (or analytic subset) of $X$.
Let $\pi\colon X'\to X$ be a resolution of singularities of $X$ contained in $\zeta$ and $Z>0$ an effective $\pi$-exceptional divisor on $X'$ with $\pi_{*}\mathcal{I}_Z\subset \mathcal{I}_{\zeta}$.
Let $\Delta$ be the {\em anti-canonical cycle} of $\pi$, namely the $\pi$-exceptional $\Q$-divisor defined by $\Delta=\pi^{*}K_X-K_{X'}$. 
Then we define the number $\delta_{\zeta}(\pi, Z)$ to be $0$ if $\Delta-Z$ is effective, and $-(\Delta-Z)^2$ otherwise.

For a cluster $\zeta$ and an effective divisor $D>0$ on $X$, we say that the above pair $(\pi,Z)$ {\em satisfies the condition $(E)_{D,\zeta}$} if $\pi_{*}\mathcal{I}_Z\subset \mathcal{I}_{\zeta}$ and $\pi^{*}D+\Delta-Z$ is effective.
\end{definition}

The first main theorem in this section is as follows:

\begin{theorem}[Reider-type theorem I]\label{ReithmI}
Let $X$ be a normal proper surface over a field $k$ or analytic in Fujiki's class $\mathcal{C}$.
Let $D>0$ be an effective divisor on $X$ 
and assume there is a chain-connected component $D_{\mathrm{c}}$ of $D$ containing a prime divisor $C$ with $H^0(\O_C)\cong k$.
Let $\zeta$ be a cluster on $X$ along which $K_X+D$ is Cartier.
Let $(\pi,Z)$ be a pair satisfying the condition $(E)_{D_{\mathrm{c}},\zeta}$ in Definition~\ref{delta} and $D':=\pi^{*}D_{\mathrm{c}}+\Delta-Z$. 
Assume that $H^0(\O_X(K_X+D))\to H^{0}(\O_X(K_X+D)|_{\zeta})$ is not surjective.
Then there exists an effective decomposition $D=A+B$ 
such that both $A$ and $B$ intersect $\zeta$ and $AB\le \frac{1}{4}\delta_{\zeta}(\pi,Z)$ holds if one of the following holds:

\smallskip

\noindent
$(\mathrm{i0})$ $\Char k=0$ and $H^1(\O_X)\cong H^1(\O_{X'})$.

\smallskip

\noindent
$(\mathrm{ip})$ $\Char k>0$ and either $H^1(\O_{X'})=0$ or 
$H^1(\O_{X'})_n=0$ and $b_1(\Gamma(\widehat{D_{\mathrm{c}}}))=b_1(\Gamma(D_{\mathrm{c}}))$, where $\widehat{D_{\mathrm{c}}}$ is the proper transform of $D_{\mathrm{c}}$ on $X'$.

\smallskip

\noindent
$(\mathrm{ii0})$ $\Char k=0$, $\kappa(D')\ge 1$ and $|mD'|$ is not composed with irrational pencils for $m\gg 0$.

\smallskip

\noindent
$(\mathrm{iip})$ $\Char k>0$, $H^1(\O_{X'})_n=0$ and $D'$ is big.

\smallskip

\noindent
$(\mathrm{iiip})$ $k$ is perfect with $\Char k>0$, $\dim |D'|\ge \dim H^1(\O_{X'})_n$ and $D'$ is big.
\end{theorem}

\begin{proof}
We first show the claim when $D$ is chain-connected.
Note that $H^0(\O_X(K_X+D))\to H^0(\O_X(K_X+D)|_{\zeta})$ is surjective if and only if $H^1(\mathcal{I}_{\zeta}\O_X(K_X+D))\to H^1(\O_X(K_X+D))$ is injective.
By the Leray spectral sequence
$$
E^{p,q}_2=H^p(R^{q}\pi_{*}\O_{X'}(K_{X'}+D'))\Rightarrow E^{p+q}=H^{p+q}(\O_{X'}(K_{X'}+D'))
$$
and the assumption $\pi_{*}\mathcal{I}_{Z}\subset \mathcal{I}_{\zeta}$, 
the injectivity of $H^1(\mathcal{I}_{\zeta}\O_X(K_X+D))\to H^1(\O_X(K_X+D))$ follows from that of $H^1(\O_{X'}(K_{X'}+D'))\to H^1(\O_{X'}(K_{X'}+D'+Z))$.
By the Serre duality, the injectivity of $H^1(\O_{X'}(K_{X'}+D'))\to H^1(\O_{X'}(K_{X'}+D'+Z))$ is equivalent to the surjectivity of $H^1(\O_{X'}(-D'-Z))\to H^{1}(\O_{X'}(-D'))$.
If $\alpha(X',D'+Z)=\alpha(X',D')$ holds,
then $H^1(\O_{X'}(-D'-Z))\to H^{1}(\O_{X'}(-D'))$ is surjective  if and only if $H^0(\O_{D'+Z})\to H^0(\O_{D'})$ is surjective because of the following exact sequences
$$
\xymatrix{
 0 \ar[r] & H^0(\O_{X'}) \ar[r] \ar@{=}[d] & H^0(\O_{D'+Z}) \ar[d] \ar[r] \ar[d] & H^1(\O_{X'}(-D'-Z)) \ar[d] \ar[r] & \alpha(X',D'+Z) \ar[r] \ar[d] & 0 \\
0 \ar[r] & H^0(\O_{X'}) \ar[r] & H^0(\O_{D'}) \ar[r] & H^1(\O_{X'}(-D')) \ar[r] & \alpha(X',D') \ar[r] & 0.
}
$$

Now we assume that $H^0(\O_X(K_X+D))\to H^{0}(\O_X(K_X+D)|_{\zeta})$ is not surjective.
Note that any one of the conditions (i0), (ip), (ii0) and (iip) implies $\alpha(X',D'+Z)=\alpha(X',D')$ by Propositions~\ref{surfalphavan} and \ref{surfpencilvan} and Corollary~\ref{Francia}.
Thus the above argument implies that $H^0(\O_{D'+Z})\to H^0(\O_{D'})$ is not surjective.
Hence $D'$ is not chain-connected by Lemma~\ref{ccconn}.
When the condition (iiip) holds, Theorem~\ref{bigzposvan} implies that $D'$ is also not chain-connected.
Therefore $D'$ decomposes into $D'=A'+B'$ such that $A'$ is chain-connected and contains the proper transform $\widehat{C}$ of $C$, $B'$ is non-zero effective and $-A'$ is nef over $B'$ by Lemma~\ref{ccc}. 
Let us define $A:=\pi_{*}A'$ and $B:=\pi_{*}B'$.
Now we show that $B$ is non-zero and intersects $\zeta$.
If $B=0$, then $B'$ is $\pi$-exceptional.
Replacing $Z$ to $Z+B'$, it follows from the same argument as above that $H^0(\O_{D'+Z})\to H^0(\O_{D'-B'})$ is not surjective,
which contradicts the chain-connectivity of $A'=D'-B'$.
If $B\cap \zeta=\emptyset$, then for any prime component $E\le B$, the proper transform $\widehat{E}$ coincides with the Mumford pull-back $\pi^{*}E$.
Since $-A'$ is nef over $B'$, we have $-AE=-A'\pi^{*}E=-A'\widehat{E}\ge 0$, which contradicts that $D$ is chain-connected.
Similarly, one can see that $A$ intersects $\zeta$.
Suppose that $\delta_{\zeta}(\pi,Z)=0$, that is, $\Delta-Z$ is effective.
Then we may assume that $D'=\ulcorner \pi^{*}D \urcorner$ by replacing $Z$ to the effective divisor $\Delta-\{-\pi^{*}D\}$.
It follows from Proposition~\ref{ccpullback} that $D'$ is chain-connected, which is a contradiction.
Hence we have $\delta_{\zeta}(\pi,Z)>0$.
Let us write $B'=\pi^{*}B+B_{\pi}$ for some $\pi$-exceptional $\Q$-divisor $B_{\pi}$ on $X'$.
Since $-A'$ is nef over $B'$, we have
$$
0\le -A'B'=-AB+(B_{\pi}-\Delta+Z)B_{\pi}.
$$
Thus we obtain
$$
AB\le \left(B_{\pi}-\frac{\Delta-Z}{2}\right)^2+\frac{\delta_{\zeta}(\pi,Z)}{4}\le \frac{\delta_{\zeta}(\pi,Z)}{4},
$$
which completes the proof for the case that $D$ is chain-connected.

For a general $D$, we consider the effective decomposition $D=D_1+D_2$, where $D_1:=D_{\mathrm{c}}$ is the chain-connected component of $D$ containing $C$ and $D_2:=D-D_1$.
Then $-D_1$ is nef over $D_2$.
If $D_2$ intersects $\zeta$, then $A:=D_1$ and $B:=D_2$ satisfy the assertion of the theorem.
Thus we may assume that $D_2$ and $\zeta$ are disjoint.
Then $H^{0}(\O_X(K_X+D_1))\to H^{0}(\O_X(K_X+D_1)|_{\zeta})$ is also not surjective.
As shown in the first half of the proof, we can take an effective decomposition $D_1=A_1+B_1$ such that both $A_1$ and $B_1$ intersect $\zeta$ and $A_1B_1\le \delta_{\zeta}(\pi,Z)/4$.
If $D_2B_1\le 0$, then we have $(D-B_1)B_1\le A_1B_1\le \delta_{\zeta}(\pi,Z)/4$.
Thus $A:=A_1+D_2$ and $B:=B_1$ satisfy the claim.
If $D_2B_1>0$, then one can see that $A:=A_1$ and $B:=B_1+D_2$ satisfy the claim.
\end{proof}

\begin{remark} \label{Reithmrem}
(1) The condition $H^0(\O_C)\cong k$ in Theorem~\ref{ReithmI} is used in the proof only to ensure that the chain-connectivity of $D'$ implies $H^0(\O_{D'})\cong k$.
Hence if we assume that $D'$ is big and $X$ is geometrically connected over a perfect field $k$, then the assumption $H^0(\O_C)\cong k$ is not needed by Corollary~\ref{bigzposinsep}.

\smallskip

\noindent
(2) In the situation of Theorem~\ref{ReithmI}, we further assume that $\pi_{*}\mathcal{I}_{Z}=\mathcal{I}_{\zeta}$ and $R^{1}\pi_{*}\mathcal{I}_{Z}=0$ (e.g., $\Delta-Z$ is $\pi$-$\Z$-positive).
Then the above proof of the theorem says that $H^0(\O_X(K_X+D))\to H^{0}(\O_X(K_X+D)|_{\zeta})$ is surjective if and only if $H^0(\O_{D'+Z})\to H^0(\O_{D'})$ is surjective, where $D':=\pi^{*}D+\Delta-Z$.
This is a generalization of Theorem~2 in \cite[p.144]{Fra}.

\smallskip

\noindent
(3) If we further assume that $D_{\mathrm{c}}^{2}>\delta_{\zeta}(\pi,Z)$ in Theorem~\ref{ReithmI}, then $2A-D_{\mathrm{c}}$ is automatically big by the construction of $A$ and $B$ (cf.\ Theorem~5.2 in \cite{Eno}).
\end{remark}

Next we give a variant of Reider-type theorems which is used in Section~6.

\begin{definition} \label{qdef}
(1) Let $X$ be a normal complete surface and $\zeta$ a cluster on $X$.
We define the invariants $q_X$ and $q_{X,\zeta}$ as 
$$
q_X:=\min\{E^2\ |\ \text{$E$ is an effective divisor on $X$ with $E^2>0$}\},
$$

$$
q_{X,\zeta}:=\min\{E^2\ |\ \text{$E$ is an effective divisor on $X$ with $E\cap \zeta \neq \emptyset$ and  $E^2>0$}\}.
$$

\smallskip

\noindent
(2) Let us define a function $\mu\colon \R_{>0}^2\to \R$ as 
$$
\mu(x,d):=\min\{x,d\}\left(\frac{d}{\min\{x,d\}}+1\right)^2.
$$
Note that $\mu(-,d)$ is a non-increasing function which takes the minimum value $4d$  and $\mu(x,-)$ is monotonically increasing for any fixed numbers $x$ and $d$.
Let $\zeta$ and $(\pi,Z)$ be as in Definition~\ref{delta}.
We define  the number $\delta'_{\zeta}(\pi,Z)$ as
$$
\delta'_{\zeta}(\pi,Z):=\mu(q_{X,\zeta},\frac{1}{4}\delta_{\zeta}(\pi,Z)).
$$
Note that $\delta'_{\zeta}(\pi,Z)\ge \delta_{\zeta}(\pi,Z)$ with the equality holding if and only if $q_{X,\zeta}\ge \frac{1}{4}\delta_{\zeta}(\pi,Z)$.
\end{definition}

The second main theorem in this section is a positive characteristic analog of Theorem~5.4 in \cite{Eno}.

\begin{theorem}[Reider-type theorem II] \label{ReithmII}
Let $X$ be a normal geometrically connected proper surface over a perfect field $k$ of positive characteristic.
Let $D$ be an effective and nef divisor on $X$.
Let $\zeta$ be a cluster on $X$ along which $K_X+D$ is Cartier.
Let $(\pi,Z)$ be a pair satisfying the condition $(E)_{D,\zeta}$ in Definition~\ref{delta}.
We assume that 
$D^2>\delta'_{\zeta}(\pi,Z)$ $($resp.\ $D^2=\delta'_{\zeta}(\pi,Z)>\delta_{\zeta}(\pi,Z)$$)$ and $\dim|D'|\ge \dim H^1(\O_{X'})_n$, where $D':=\pi^{*}D+\Delta-Z$.
If $H^0(\O_X(K_X+D))\to H^{0}(\O_X(K_X+D)|_{\zeta})$ is not surjective, then
there exists an effective decomposition $D=A+B$ with $A, B>0$ intersecting $\zeta$ such that $A-B$ is big, $B$ is negative semi-definite and $AB\le \frac{1}{4}\delta_{\zeta}(\pi,Z)$ $($resp.\ or $B^2=q_{X,\zeta}$ and $D\equiv \left(\frac{\delta_{\zeta}(\pi,Z)}{4q_{X,\zeta}}+1\right)B$$)$.
\end{theorem}

\begin{proof}
Note that $D':=\pi^{*}D+\Delta-Z$ is automatically big since $D^2>\delta_{\zeta}(\pi,Z)$.
Thus the assumption of the theorem implies the condition (iiip) in Theorem~\ref{ReithmI}.
Hence there exists an effective decomposition $D=A+B$ such that both $A$ and $B$ intersect $\zeta$, $A-B$ is big and $AB\le \delta_{\zeta}(\pi,Z)/4$, where we note that $D$ is chain-connected from Lemma~\ref{nefbignumconn}.
The rest of the proof is similar to that of Theorem~5.4 in \cite{Eno}.
\end{proof}

\subsection{Corollaries of Reider-type theorems}

In this subsection, we collect corollaries of Theorems~\ref{ReithmI} and \ref{ReithmII}.
For simplicity, the base field $k$ is assumed to be algebraically closed.
First we consider the criterion of the basepoint-freeness.

\begin{definition} \label{deltatau}
Let $X$ be a normal proper surface and $x\in X$ a closed point.
Let $\pi\colon X'\to X$ be the blow-up at $x$ if $x\in X$ is smooth, or the minimal resolution at $x$ otherwise.
Let $Z>0$ denote the exceptional $(-1)$-curve (resp.\ the fundamental cycle of $\pi$, the round-up $\ulcorner \Delta\urcorner$, the round-down $\llcorner \Delta\lrcorner$) if $x\in X$ is smooth (resp.\ Du Val, log terminal but not Du Val, not log terminal).
We simply write by $\delta_x$ the number $\delta_x(\pi,Z)$ in Definition~\ref{delta}.
Then we define the number $\tau_x$ to be $3$ (resp.\ $1$, $\dim V_n$) if $x\in X$ is smooth (resp.\ Du Val, otherwise), where $V_n$ is the nilpotent part of the $k$-vector space $V:=(R^1\pi_{*}\O_{X'})_{x}$ under the Frobenius action.
\end{definition}

The following lemma is easy.

\begin{lemma} \label{deltataulem}
Let the situation be as in Definition~\ref{deltatau} and $D$ an effective divisor on $X$ passing through $x$ such that $K_X+D$ is Cartier at $x$.
Then the following hold:

\smallskip

\noindent
$(1)$ $\delta_x=4$ $($resp.\ $\delta_x=2$, $0<\delta_x<2$, $\delta_x=0$$)$ if $x\in X$ is smooth $($resp.\ Du Val, log terminal but not Du Val, not log terminal$)$.

\smallskip

\noindent
$(2)$ $\dim |D'|-\dim H^1(\O_{X'})_n+\tau_{x}\ge \dim |D|-\dim H^1(\O_{X})_n$,
 where $D':=\pi^{*}D+\Delta-Z$.
\end{lemma}

\begin{proof}
In order to prove (1),
we may assume that $x\in X$ is a log terminal singularity.
Let $\Delta=\sum_{i}a_iE_i$ and $Z=\sum_{i}b_iE_i$ denote the irreducible decompositions.
Then we have
\begin{align*}
(\Delta-Z)^2&=(\Delta-Z)\Delta-(\Delta-Z)Z \\
&=\sum_{i}(a_i-b_i)E_i(-K_{X'})+(K_{X'}+Z)Z \\
&=\sum_{i}(b_i-a_i)K_{X'}E_i+2p_a(Z)-2 \\
&\ge -2,
\end{align*}
and it is easy to see that the equality holds if and only if $x\in X$ is Du Val.
The claim (2) follows from the exact sequence $0\to H^1(\O_{X})_n\to H^1(\O_{X'})_n\to V_n$ induced by the Leray spectral sequence with the Frobenius action.
\end{proof}

Theorem~\ref{ReithmI} for (i0) and (ip) and Lemma~\ref{deltataulem} imply the following criterion of basepoint-freeness:

\begin{corollary} \label{bpfcor}
Let $X$ be a normal proper surface. 
Let $x\in X$ be at most a rational singularity.
Let $L$ be a divisor on $X$ which is Cartier at $x$.
We assume that there exists a chain-connected member $D\in |L-K_X|$ passing through $x$ satisfying the following conditions:

\smallskip

\noindent
$(\mathrm{i})$ $(X,x)$ or $(D,x)$ is singular, 

\smallskip

\noindent
$(\mathrm{ii})$ $D$ is strictly $(\delta_x/4)$-connected if $x\in X$ is log terminal, and

\smallskip

\noindent
$(\mathrm{iii})$
The Frobenius map on $H^1(\O_X)$ is injective and $b_1(\Gamma(D))=b_1(\Gamma(\widehat{D}))$ when $\Char k>0$, where $\widehat{D}$ is the proper transform of $D$ by the minimal resolution of $(X,x)$ when $(X,x)$ is singular or by the blow-up at $x$ when $(X,x)$ is smooth.

\smallskip

\noindent
Then $x$ is not a base point of $|L|$.
\end{corollary}

\begin{remark}
All the conditions of $D$ in Corollary~\ref{bpfcor} are satisfied if $D$ is an integral curve passing through $x$, $(X,x)$ or $(D,x)$ is singular, and $D$ is analytically irreducible at $x$ when $\Char k>0$.
\end{remark}

Theorem~\ref{ReithmI} for (iiip) and Lemma~\ref{deltataulem} imply the following corollary:

\begin{corollary} \label{bpfcor2}
Let $X$ be a normal proper surface.
Let $x\in X$ be a closed point.
Let $D$ be a nef divisor on $X$ such that $K_X+D$ is Cartier at $x$.
Then $x$ is not a base point of $|K_X+D|$ if the following conditions hold:
\smallskip

\noindent
$(\mathrm{i})$ There exist rational numbers $\alpha$ and $\beta$ with $\alpha\ge \delta_{x}$ and $4\beta(1-\beta/\alpha)\ge \delta_x$ such that
$D^2>\alpha$ and $DB\ge \beta$ for any curve $B$ on $X$ passing through $x$, and

\smallskip

\noindent
$(\mathrm{ii})$
$\dim |D|\ge \dim H^1(\O_X)_n+\tau_x$ when $\Char k>0$.
\end{corollary}

\begin{proof}
Assume contrary that $x$ is a base point of $|K_X+D|$.
By Theorem~\ref{ReithmI} and Remark~\ref{Reithmrem}~(3) (or Theorem~5.2 in \cite{Eno} when $\Char k=0$), 
there exists a curve $B$ on $X$ passing through $x$ such that $(D-B)B\le \delta_x/4$ and $D-2B$ is big.
It follows from the Hodge index theorem that
$$
DB\le \frac{1}{4}\delta_x+B^2\le \frac{1}{4}\delta_x+\frac{(DB)^2}{D^2},
$$
that is, $(DB)^2-D^2(DB)+D^2 \delta_x/4\ge 0$.
Since $(D-2B)D>0$, we obtain $DB\le (D^2-\sqrt{D^2(D^2-\delta_x)})/2$.
It follows from $D^2>\alpha$ and $DB\ge  \beta$ that 
$$
\beta\le DB\le \frac{D^2-\sqrt{D^2(D^2-\delta_x)}}{2}<\frac{\alpha-\sqrt{\alpha(\alpha-\delta_x)}}{2}.
$$
Thus we have $4\beta(1-\beta/\alpha)< \delta_x$, which contradicts the assumption (i). 
\end{proof}

The very ample cases can be obtained similarly.

\begin{corollary} \label{vacor}
Let $X$ be a normal proper surface with at most Du Val singularities.
Let $D$ be a Cartier divisor on $X$.
Then $|K_X+D|$ is very ample if the following conditions hold:
\smallskip

\noindent
$(\mathrm{i})$ There exist rational numbers $\alpha$ and $\beta$ with $\alpha\ge 8$ and $\beta(1-\beta/\alpha)\ge 2$ such that
$D^2>\alpha$ and $DB\ge \beta$ for any curve $B$ on $X$, and
\smallskip

\noindent
$(\mathrm{ii})$
$\dim |D|\ge \dim H^1(\O_X)_n+6$ when $\Char k>0$.
\end{corollary}

\begin{proof}
It suffices to show that $|K_X+D|$ separates any cluster $\zeta$ on $X$ such that it is of length $2$ and its support contains at least one smooth point of $X$ or 
the defining ideal is $\mathfrak{m}^{2}_{x}\subset \O_X$ for some Du Val singularity $x\in X$.
Now we show this when the support of $\zeta$ has a single point $x\in X$ (the case $\zeta=x+y$, $x\neq y$ is similar).

First we assume that $x\in X$ is smooth.
By assumption, $\zeta$ is a tangent vector at $x$.
Let $\pi \colon X'\to X$ be the blow-up along $\zeta$, that is, the composition of the blow-ups at $x$ and at the point infinitely near to $x$ corresponding to the tangent vector.
Let $Z$ be the sum of the total transforms of the two exceptional $(-1)$-curves.
Then one can see that $\delta_{\zeta}(\pi,Z)=8$ and $\dim|D'|\ge \dim |D|-6$, where $D':=\pi^{*}D+\Delta-Z$.

We assume that $x\in X$ is Du Val.
Let $\pi\colon X'\to X$ be the minimal resolution of $x$ and $E$ its fundamental cycle.
Putting $Z:=2E$, one can see that $\pi_{*}\O_{X'}(-Z)=\mathfrak{m}^{2}_{x}$, $\delta_{\zeta}(\pi,Z)=8$ and $\dim|D'|\ge \dim |D|-4$.

The rest of the proof is similar to that of Corollary~\ref{bpfcor2}.
\end{proof}

\begin{remark}
One can obtain the similar result of Corollary~\ref{vacor} when $X$ is not necessarily canonical (but need the estimation of $\delta_{\zeta}(\pi,Z)$).
For the direction, see \cite{Sak}, \cite{KaMa}, \cite{LanII}.
\end{remark}

The following is a partial answer to the Fujita conjecture for surfaces in positive characteristic (although there are counter-examples to the Fujita conjecture \cite{GZZ}).

\begin{corollary}\label{Fujitacor}
Let $X$ be a projective surface with at most Du Val singularities in positive characteristic.
Let $H$ be an ample Cartier divisor on $X$.
Then $|K_X+mH|$ is base point free for any $m\ge 3$ $($or $m=2$ and $H^2>1$$)$ with $\dim|mH|\ge 3+\dim H^1(\O_X)_n$, and is very ample for any $m\ge 4$ $($or $m=3$ and $H^2>1$$)$ with $\dim |mH|\ge 6+\dim H^1(\O_X)_n$.
\end{corollary}

\begin{proof}
This follows from Corollary~\ref{bpfcor2} with $(\alpha,\beta)=(4,2)$ and Corollary~\ref{vacor} with $(\alpha,\beta)=(9,3)$ (or directly from Lemma~\ref{ampleconn} and Theorem~\ref{ReithmI}).
\end{proof}

For pluri-(anti)canonical maps, the following hold:

\begin{corollary} \label{pluricor}
Let $X$ be a normal projective surface with at most singularities of geometric genera $p_g\le 3$ in positive characteristic.
Then the following hold:

\smallskip

\noindent
$(1)$ If $K_X$ is ample Cartier, then $|mK_X|$ is base point free for $m\ge 4$ $($or $m=3$ and $K_{X}^{2}>1$$)$ with $\dim|(m-1)K_X|\ge \dim H^1(\O_X)_n+3$.

\smallskip

\noindent
$(2)$ If $X$ is canonical and $K_X$ is ample Cartier, then $|mK_X|$ is very ample for any $m\ge 5$ $($or $m=4$ and $K_{X}^{2}>1$$)$ with $\dim |(m-1)K_X|\ge \dim H^1(\O_X)_n+6$.

\smallskip

\noindent
$(3)$ If $-K_X$ is ample Cartier $($that is, $X$ is canonical del Pezzo$)$, then $|-mK_X|$ is base point free for $m\ge 2$ $($or $m=1$ and $K_{X}^{2}>1$$)$, and is very ample for any $m\ge 3$ $($or $m=2$ and $K_{X}^{2}>1$$)$.

\smallskip

\noindent
$(4)$ If $K_X$ is ample with Cartier index $r\ge 2$, then $|mrK_X|$ is base point free for $m\ge 3$ with $\dim|(mr-1)K_X|\ge \dim H^1(\O_X)_n+3$.

\smallskip

\noindent
$(5)$ If $-K_X$ is ample with Cartier index $r\ge 2$ $($that is, $X$ is klt del Pezzo$)$, then $|-mrK_X|$ is base point free for $m\ge 2$ with $\dim|-(mr+1)K_X|\ge 3$.
\end{corollary}

\begin{proof}
This follows from Corollary~\ref{bpfcor2} and Corollary~\ref{Fujitacor}.
Note that $\dim|-(m+1)K_X|\ge 6$ automatically holds in the case (3) by the Riemann--Roch theorem
and that any klt del Pezzo surface is rational and hence $H^1(\O_X)=0$ by Theorem~3.5 in \cite{Tana}.
\end{proof}

\begin{remark}
(1) When $X$ is canonical, Corollary~\ref{pluricor} (1) and (2) were obtained by {\cite[Main theorem p.97]{Eke}} without the condition for $\dim|(m-1)K_X|$.

\smallskip

\noindent
(2) Corollary~\ref{pluricor} (3) was shown in {\cite[Proposition~2.14]{BeTa}}.

\end{remark}

For bicanonical maps on smooth surfaces of general type, the following can be shown (compare Theorem~26 and Theorem~27 in \cite{S-B}):

\begin{corollary}
Let $X$ be a smooth minimal projective surface of general type in positive characteristic.
Then $|2K_X|$ is base point free if $K_{X}^{2}>4$ and $\chi(\O_X)\ge 5-h^{0,1}_{s}$, and $|2K_X|$ defines a birational morphism if $K_{X}^{2}>9$, $\chi(\O_X)\ge 8-h^{0,1}_{s}$ and $X$ does not admit genus $2$ fibrations, where $h^{0,1}_{s}$ is the dimension of the semi-simple part $H^1(\O_X)_{s}$.
\end{corollary}

\begin{proof}
The base point free case follows from Theorem~\ref{ReithmI} and the $2$-connectedness of $K_X$ (Lemma~1 in \cite{Bom}).
Note that $\dim|K_X|\ge \dim H^1(\O_X)_n+3$ is equivalent to $\chi(\O_X)\ge 5-h^{0,1}_{s}$.
Next we consider the birational case.
If the bicanonical map is not birational (hence generically finite of degree $\ge 2$), there exist infinitely many clusters $\zeta=x+y$ of degree $2$ with $x\neq y$ such that $|2K_X|$ does not separate $\zeta$.
One can see easily that $\delta_{\zeta}(\pi,Z)=8$ and $\delta'_{\zeta}(\pi,Z)=8$ or $9$, where $\pi\colon X'\to X$ is the blow-up along $\zeta=x+y$ and $Z=E_x+E_y$ is the sum of two exceptional $(-1)$-curves.
It follows from Theorem~\ref{ReithmII} that there exists a negative semi-definite curve $B_{\zeta}$ intersecting $\zeta$ such that $(K_X-B_{\zeta})B_{\zeta}=2$, where the equality is due to the $2$-connectivity of $K_X$.
Thus we have $(K_XB_{\zeta},B_{\zeta}^{2})=(2,0)$ or $(1,-1)$.
Since the number of curves $B_{\zeta}$ satisfying the later case is finite, there exist infinitely many curves (but belong to finitely many numerical classes) $B_{\zeta}$ satisfying the former case.
By applying Proposition~6.7 in \cite{Eno} (see Lemma~\ref{fibrlem}), these $B_{\zeta}$ define a genus $2$ fibration on $X$.
\end{proof}

\section{Extension theorems in positive characteristic}
\label{sec:Extension theorems in positive characteristic}

In this section, let $X$ stand a normal proper geometrically connected surface over an infinite perfect field $k$ of positive characteristic. 
We will prove the following extension theorem, which is a positive characteristic analog of Theorem~6.1 in \cite{Eno}, by using Theorem~\ref{ReithmII} instead of Theorem~5.4 in \cite{Eno}.

\begin{theorem}[Extension theorem] \label{extnthm}
Let $D>0$ be an effective divisor on $X$ and assume that any prime component $D_i$ of $D$ has positive self-intersection number.
Let $\varphi\colon D\to \mathbb{P}^1$ be a finite separable morphism of degree $d$. 
If $D^2>\mu(q_X,d)$ and $\dim |D|\ge 3d+\dim H^1(\O_X)_n$, then
there exists a morphism $\psi\colon X\to \mathbb{P}^1$ such that $\psi|_D=\varphi$.
\end{theorem}

\begin{remark}
Theorem~\ref{extnthm} is a generalization of Serrano's result (Remark~3.12 in \cite{Ser}).
Paoletti proved another variant of extension theorems in positive characteristic by using Bogomolov-type inequalities (Theorem~3.1 in \cite{Pao}).
\end{remark}

The following two theorems are positive characteristic analogs of Theorem~6.10 and Theorem~6.11 in \cite{Eno} (see Section~6 in \cite{Eno} for notations and discussions).

\begin{theorem}[Extension theorem with base points] \label{bextnthm}
Let $D>0$ be an effective divisor on $X$ and assume that any prime component $D_i$ of $D$ has positive self-intersection number.
Let $\varphi\colon D\to \mathbb{P}^1$ be a finite separable morphism of degree $d$ which can not be extended to a morphism on $X$.
We assume that $D^2=\mu(q_X,d)$, $q_X<d$ and $\dim |D|\ge 3d+\dim H^1(\O_X)_n$.
Then there exists a linear pencil $\{F_{\lambda}\}_{\lambda}$ with $F_{\lambda}^2=q_{X}$ and no fixed parts such that the induced rational map $\psi\colon X\dasharrow \mathbb{P}^1$ satisfies $\psi|_D=\varphi$.
\end{theorem}

\begin{theorem}[Extension theorem on movable divisors] \label{mextnthm}
Let $D>0$ be an effective divisor on $X$ and assume that all prime components $D_i$ of $D$ have non-trivial numerical linear systems and positive self-intersection numbers.
Let $\varphi\colon D\to \mathbb{P}^1$ be a finite separable morphism of degree $d$  on $X$.
If $D^2>\mu(q_{X,\infty},d)$ $($resp.\ $D^2=\mu(q_{X,\infty},d)$, $q_{X,\infty}<d$$)$ and $\dim |D|\ge 3d+\dim H^1(\O_X)_n$,
then 
there exists a morphism $\psi\colon X\to \mathbb{P}^1$ $($resp.\ or a rational map $\psi\colon X\dasharrow \mathbb{P}^1$ induced by a linear pencil $\{F_{\lambda}\}_{\lambda}$ with $F_{\lambda}^2=q_{X,\infty}$ and no fixed parts$)$ such that $\psi|_D=\varphi$.
\end{theorem}

\subsection{The proof of Extension theorem}

The proofs of Theorems~\ref{extnthm}, \ref{bextnthm} and \ref{mextnthm} are almost identical to that of Theorems~6.1, 6.10 and 6.11 in \cite{Eno}.
We only sketch here the proof of Theorem~\ref{extnthm}
(the remaining cases are left to the reader).

Let $\Lambda$ be the set of closed points of $\mathbb{P}^1$ such that 
 $(\varphi|_{D_{\mathrm{red}}})^{-1}(\lambda)$ is reduced
and contained in the smooth loci of $X$ and $D_{\mathrm{red}}$ .
It is a dense subset of $\mathbb{P}^1$ since $X$ is normal and $\varphi|_{D}$ is separable.
For a closed point $\lambda\in \mathbb{P}^1$, we put $\mathfrak{a}_{\lambda}:=\varphi^{-1}(\lambda)$.

\begin{lemma}[\cite{Eno} Lemma~6.4]
For any $k$-rational point $\lambda\in \Lambda$, the restriction $H^0(\O_{X}(K_X+D))\to H^0(\O_X(K_X+D)|_{\mathfrak{a}_{\lambda}})$ is not surjective.
\end{lemma}

\begin{lemma}[cf.\ \cite{Eno} Lemma~6.5] \label{deltalem}
For any $k$-rational point $\lambda\in \Lambda$, there exist a member $D_{\lambda}\in |D|$ and a pair $(\pi,Z)$ satisfying the condition $(E)_{D_{\lambda},\mathfrak{a}_{\lambda}}$ in Definition~\ref{delta} such that $\delta_{\mathfrak{a}_{\lambda}}(\pi,Z)=4d$.
\end{lemma}

\begin{proof}
As in the proof of Lemma~6.5 in \cite{Eno}, we can take a blow-up $\pi\colon X'\to X$ along $\mathfrak{a}_{\lambda}$ and a $\pi$-exceptional divisor $Z>0$ such that $\pi_{*}\mathcal{I}_{Z}=\mathcal{I}_{\mathfrak{a}_{\lambda}}$ and $\delta_{\mathfrak{a}_{\lambda}}(\pi,Z)=4d$.
Since $\dim |D| \ge 3d$, we can take a member $D'_{\lambda}\in |\pi^{*}D+\Delta-Z|$.
Then $D_{\lambda}:=\pi_{*}D'_{\lambda}$ is a desired one.
\end{proof}

We fix a $k$-rational point $\lambda\in \Lambda$ arbitrarily.
Then $D_{\lambda}$, $\mathfrak{a}_{\lambda}$ and the pair $(\pi,Z)$ obtained by Lemma~\ref{deltalem} satisfy the condition $(E)_{D_{\lambda},\mathfrak{a}_{\lambda}}$.
Thus we can apply Theorem~\ref{ReithmII} to this situation since $D_{\lambda}^2=D^2>\mu(q_X,d)\ge \delta'_{\mathfrak{a}_{\lambda}}(\pi,Z)$ and $\dim |D'_{\lambda}|\ge \dim H^1(\O_{X'})_n$.
Hence there exists an effective decomposition $D_{\lambda}=A_{\lambda}+B_{\lambda}$ with $A_{\lambda}$ and $B_{\lambda}$ intersecting $\mathfrak{a}_{\lambda}$ such that $A_{\lambda}-B_{\lambda}$ is big, $B_{\lambda}$ is negative semi-definite and $A_{\lambda}B_{\lambda}\le d$.
Moreover the following lemma can be shown similarly to Lemma~6.6 in \cite{Eno}.

\begin{lemma}[\cite{Eno} Lemma~6.6] \label{fiberlem}
In the above situation, we have $B_{\lambda}^2=0$ and $D\cap B_{\lambda}\subset \mathfrak{a}_{\lambda}$ scheme-theoretically.
\end{lemma}

Let $\mathcal{B}$ be the set of all prime divisors $C$ such that $C\le B_{\lambda}$ for some $k$-rational point $\lambda\in \Lambda$ and $DC>0$.
This is an infinite set because $D\cap B_{\lambda}\subset\mathfrak{a}_{\lambda}$ by Lemma~\ref{fiberlem} and $\mathfrak{a}_{\lambda}\cap \mathfrak{a}_{\lambda'}=\emptyset$ for $\lambda\neq \lambda'$. 
On the other hand, the set $\mathcal{B}$ consists of finitely many numerical equivalence classes, say $B_{(1)},\ldots, B_{(m)}$, since $0<DC\le DB_{\lambda}\le d$ for any $C\in \mathcal{B}$.
We put $\mathcal{B}_{(i)}:=\{C\in \mathcal{B}\ |\ C\equiv B_{(i)}\}$.
Then there is at least one $\mathcal{B}_{(i)}$ which has infinite elements.
We choose such a $\mathcal{B}_{(i)}$ and put $\mathcal{B}:=\mathcal{B}_{(i)}$ again.

\begin{lemma}[\cite{Eno} Proposition~6.7] \label{fibrlem}
Let $X$ be a normal proper surface over an infinite perfect field $k$.
Let $\mathcal{B}$ be an infinite family of prime divisors on $X$ any member of which has the same numerical equivalence class $B$ with $B^2=0$.
Then there exists a fibration $f\colon X\to Y$ onto a smooth curve $Y$ such that any member of $\mathcal{B}$ is a fiber of $f$.
\end{lemma}

By using Lemma~\ref{fibrlem} in this situation, there exists a fibration $f\colon X\to Y$ onto a smooth curve $Y$ such that any member of $\mathcal{B}$ is a fiber of $f$.
Let $\overline{D}$ denote the scheme-theoretic image of $(f|_D,\varphi)\colon D\to Y\times \mathbb{P}^1$ and $h\colon \overline{D}\to Y$ the restriction of the first projection $Y\times \mathbb{P}^1\to Y$ to $\overline{D}$.

\begin{lemma}[\cite{Eno} Lemma~6.8]
$h\colon \overline{D}\to Y$ is an isomorphism.
\end{lemma}

Let $\gamma\colon Y \to\mathbb{P}^1$ be the composition of $h^{-1}$ and the second projection $\overline{D}\to \mathbb{P}^1$.
Then $\varphi\colon D\to \mathbb{P}^1$ decomposes into $\varphi=\gamma\circ f|_D\colon X\to Y\to \mathbb{P}^1$.
Hence $\psi:=\gamma\circ f$ is the desired one.

\section{Applications to plane curves}
\label{sec:Applications to plane curves}

In this section, we are going to apply our extension theorems obtained in Section~6 to the geometry of plane curves.

\begin{definition}[Strange points for plane curves]
Let $D\subset \PP^2$ be a plane curve over a field $k$.
For a $k$-rational point $x\in \PP^2$, we define the open subset $U_{D,x}$ of $D$
to be the set of points $y$ of $D$ such that the reduced plane curve $(D_{k(y)})_{\mathrm{red}}\subset \PP^{2}_{k(y)}$ is non-singular at $y$ and its tangent line $L_{y}\subset \PP^2_{k(y)}$ at $y$ does not pass through $x$.
Then we say that $x$ is {\em strange} with respect to $D$ if $U_{D,x}$ is not dense in $D$.
If $\Char k=0$, all the strange points are $k$-rational points on lines contained in $D$.
If $D$ is smooth and has strange points, then $D$ is a line or a conic with $\Char k=2$ (cf.\ \cite{Har}, Chapter~IV, Theorem~3.9).
For a non-singular $k$-rational point $x$ of $D$, we can see that $x$ is not strange if and only if the inner projection $D\to \PP^1$ from $x$ is finite and separable.
\end{definition}

\begin{theorem} \label{planethm}
Let $D\subset \PP^2$  be a plane curve of degree $m\ge 3$ over an arbitrary base field $k$.
Then there is a one-to-one correspondence between 

\smallskip

\noindent
$(\mathrm{i})$
the set of non-singular $k$-rational points of $D$ which is not strange, and 

\smallskip

\noindent
$(\mathrm{ii})$
the set of finite separable morphisms $D\to \PP^1$ of degree $m-1$ up to automorphisms of $\PP^1$.

\smallskip

\noindent
Moreover, any finite separable morphism $D\to \PP^1$ has degree greater than or equal to $m-1$.
\end{theorem}

\begin{proof}
Let $x$ be a non-singular, non-strange $k$-rational point of $D$.
Then the inner projection from $x$ defines a finite separable morphism $\mathrm{pr}_{x}\colon D\to \PP^1$ of degree $m-1$.
This correspondence $x\mapsto \mathrm{pr}_x$ defines a map from the set of (i) to that of (ii), which is injective since $m\ge 3$.
Thus, in order to prove the first claim, it suffices to show that any finite separable morphism $D\to \PP^1$ of degree $m-1$ is obtained by the inner projection from some $k$-rational point of $D$.
Let $\varphi\colon D\to \PP^1$ be such a morphism.
If the base field $k$ is algebraically closed (or infinite and perfect), then by Theorem~\ref{bextnthm} (or Theorem~6.10 in \cite{Eno} when $\Char k=0$), there exists a rational map $\psi\colon \PP^2\dasharrow \PP^1$ induced by a linear pencil of lines such that $\psi|_{D}=\varphi$.
Since $\varphi$ is a morphism, $\psi$ is nothing but an inner projection from a $k$-rational point. 
Suppose that $k$ is not algebraically closed.
Taking the base change to an algebraic closure $\overline{k}$ of $k$,
we obtain a finite separable morphism $\varphi_{\overline{k}}\colon D_{\overline{k}}\to \PP^1_{\overline{k}}$ of degree $m-1$.
From the above argument, this comes from an inner projection from a $\overline{k}$-rational point $x$ of $D_{\overline{k}}$.
Now we take a $k$-rational point $\lambda\in \PP^1$ 
and write $\varphi_{\overline{k}}^{-1}(\lambda)=\{x_1,\ldots,x_{l}\}$ as sets.
Then these points $x,x_1,\ldots,x_l$ lie on the same line $L\subset \PP^2_{\overline{k}}$.
Moreover, the set $\{x_1,\ldots,x_{l}\}$ is $G$-invariant under the Galois action $G:=\mathrm{Gal}(\overline{k}/k)$ on $\PP^2_{\overline{k}}$.
On the other hand,  each $\sigma\in G$ sends the line $L$ to another line $\sigma(L)$, which also contains $\{x_1,\ldots,x_{l}\}$.
Now we show $\sigma(L)=L$, that is, $L$ is $G$-invariant.
If $l\ge 2$, then this is clear since the line passing through fixed two points is unique.
Thus we may assume $l=1$.
Then both $L$ and $\sigma(L)$ are tangent lines at $x_1$ with multiplicity $m-1$, which implies $L=\sigma(L)$.
By taking another $k$-rational point $\lambda'$ of $\PP^1$ and using the same argument as above, we can take another $G$-invariant line $L'$ in $\PP^2_{\overline{k}}$ such that $L\cap L'=\{x\}$.
Thus $x$ is $G$-invariant and descends to a $k$-rational point $x^{G}$ of $D$.
Since $x$ is a smooth point of $D_{\overline{k}}$, so is $x^{G}$.
Since the lines $L$ and $L'$ descend to lines $L^{G}$ and $L'^{G}$ which intersects at $x^{G}$, we conclude that $\varphi\colon D\to \PP^1$ is the inner projection from $x^{G}$.
The last claim is due to Theorem~\ref{extnthm} (or Theorem~6.1 in \cite{Eno} when $\Char k=0$) since $\PP^2$ does not admit any non-constant morphism to $\PP^1$.
\end{proof}

\appendix

\section{Mumford's intersection form on a normal projective variety}
\label{sec:Mumford's intersection form on a normal projective variety}

In this appendix, we extend Mumford's intersection form on a normal surface (\cite{Mum}) to a higher dimensional variety over a field $k$.

\begin{theorem} \label{Mumint}
Let $X$ be a normal projective variety of dimension $n\ge 2$ over a field $k$.
Then there exists a multilinear form
$$
Q\colon \underbrace{\Pic(X)\times \cdots \times \Pic(X)}_{n-2}\times \Cl(X)\times \Cl(X)\to \Q,
$$
which we call Mumford's intersection form,
such that the following conditions hold:

\smallskip

\noindent
$(\mathrm{i})$ $Q$ is an extension of the usual intersection form $\Pic(X)\times \cdots \times \Pic(X)\times \Cl(X)\to \Z$,

\smallskip

\noindent
$(\mathrm{ii})$ $Q$ is symmetric with respect to the first $n-2$ terms  and the last two terms,

\smallskip

\noindent
$(\mathrm{iii})$ $Q$ is compatible with the base change to any separable field extension $k'$ of $k$, and

\smallskip

\noindent
$(\mathrm{iv})$  If $k$ is an infinite field and $S:=H_1\cap \cdots \cap H_{n-2}$ is a normal surface obtained by the intersection of $n-2$ general hyperplanes, then 
$Q(H_1,\ldots,H_{n-2},D_1,D_2)$ coincides with the Mumford's intersection number of
$D_{1}|_S$ and $D_{2}|_S$ on $S$.
\end{theorem}

\begin{definition}[Mumford pull-back]
Let $X$ be a normal projective variety over an infinite field $k$.
Let $\pi\colon X'\to X$ be a resolution of $X$.
Let $\{E_i\}_i$ denote the set of $\pi$-exceptional prime divisors on $X'$ such that the center $C_i:=\pi(E_i)$ is of codimension $2$.
For each $i$, let $F_i$ denote the numerical equivalence class of $[k(x):k]^{-1}(\pi|_{E_i})^{-1}(x)$, 
where it is independent of the choice of a closed point $x\in C_i$.
For a Weil divisor $D$ on $X$, we define the {\em Mumford pull-back} of $D$ by $\pi$, which is denoted by $\pi^{\star}D$, 
as $\widehat{D}+\sum_{i}d_iE_i$, where $\widehat{D}$ is the proper transform of $D$ on $X'$ and  the coefficients $d_i\in \Q$ are determined by the equation
$(\widehat{D}+\sum_{i}d_iE_i)F_j=0$ for each $j$.
\end{definition}

\begin{remark}
(1) The definition of the Mumford pull-back makes sense if the intersection matrix $(E_iF_j)_{i,j}$ is invertible.
The invertibility can be checked as follows:
Let $H_1, \ldots, H_{n-2}$ be general hyperplanes on $X$ such that $S:=H_1\cap \cdots \cap H_{n-2}$ is a normal surface.
Let $\rho\colon S'\to \pi^{-1}(S)$ denote the normalization 
and $E'_{i}$ the pull-back of $E_i$ under $\rho$.
Then $E'_{i}$ is a non-zero effective $(\pi \circ \rho)$-exceptional divisor on $S'$ and thus $(E'_{i}E'_{j})_{i,j}$ is negative definite.
Since $E'_{i}E'_{j}=\pi^{*}H_1\cdots \pi^{*}H_{n-2}E_iE_j=(H_1\cdots H_{n-2}C_j)E_iF_j$, the matrix $(E_iF_j)_{i,j}$ is invertible.

\smallskip

\noindent
(2) The definition of the Mumford pull-back seems to be unnatural because all the coefficients of $\pi$-exceptional divisors contracting to codimension $\ge 3$ centers are zero.
It seems to be natural to consider that the Mumford pull-back is determined modulo $\pi$-exceptional divisors contracting to codimension $\ge 3$ centers.
Indeed, the terms of such $\pi$-exceptional divisors do not affect the intersection numbers of $n-2$ Cartier divisors and two Weil divisors defined later.
For more general treatment of Mumford pull-backs, see \cite{BdFF}.
\end{remark}

\begin{definition} \label{altintnumber}
Let $X$ be a normal projective variety of dimension $n\ge 2$ over an infinite field $k$.
Let $L_1,\ldots,L_{n-2}$ be Cartier divisors on $X$.
Let $D_1$ and $D_2$ be Weil divisors on $X$.
\smallskip

\noindent
(1) For a resolution $\pi\colon X'\to X$ of $X$,
we define $(L_1\cdots L_{n-2}D_{1}D_{2})_{\pi}$ to be the rational number $\pi^{*}L_1\cdots \pi^{*}L_{n-2}\pi^{\star}D_{1}\pi^{\star}D_{2}$.

\smallskip

\noindent
(2) Let $\pi\colon Y'\to X$ be an alteration from a regular projective variety $Y'$.
Let $Y'\xrightarrow{\psi} Y \xrightarrow{\varphi} X$ denote the Stein factorization of $\pi$,
where $\psi$ is a resolution of a normal projective variety $Y$ and $\varphi$ is finite.
Then we define 
$$
(L_1\cdots L_{n-2}D_{1}D_{2})_{\pi}:=\frac{1}{\deg \varphi}(\varphi^{*}L_1\cdots \varphi^{*}L_{n-2}\varphi^{*}D_{1}\varphi^{*}D_{2})_{\psi}.
$$
\end{definition}

\begin{lemma}
Let $X$, $L_1\ldots,L_{n-2}$, $D_{1}$, $D_{2}$ be as in Definition~\ref{altintnumber}.
Then the numbers $(L_1\cdots L_{n-2}D_{1}D_{2})_{\pi}$ are independent of alterations $\pi$.
\end{lemma}

\begin{proof}
For two alterations $\pi_{i}\colon Y'_{i}\to X$ with $Y'_{i}$ regular, $i=1,2$, we can take an alteration $Y'_{3}\to Y'_{1}\times_{X} Y'_{2}$ with $Y'_{3}$ regular (\cite{dJo}).
Thus we may assume that there exists a generically finite morphism $\rho\colon Y'_{2}\to Y'_{1}$ such that $\pi_{1}\circ \rho =\pi_{2}$.
Let $Y'_{i}\xrightarrow{\psi_{i}} Y_{i} \xrightarrow{\varphi_{i}} X_{i}$ denote the Stein factorization of $\pi_{i}$.
Then there exists a finite morphism $\tau\colon Y_2\to Y_1$ such that $\varphi_2=\varphi_1\circ \tau$.
Then it suffices to show that for any Weil divisor $D$ on $Y_1$,
$\rho^{*}\psi_{1}^{\star}D$ equals $\psi_{2}^{\star}\tau^{*}D$ modulo $\psi_{2}$-exceptional divisors contracting to codimension $\ge 3$ centers.
One can check this by direct computations.
\end{proof}

\begin{definition}[Intersection numbers]
Let $X$ be a normal projective variety of dimension $n\ge 2$ over a field $k$.
Let $L_1,\ldots,L_{n-2}$ be Cartier divisors on $X$.
Let $D_1$ and $D_2$ be Weil divisors on $X$.
Then we define the {\em intersection number} of $L_1,\ldots,L_{n-2},D_1$ and $D_2$, which is denoted by $L_1\cdots L_{n-2}D_{1}D_{2}$, as follows:

\smallskip

\noindent
(1) If the base field $k$ is infinite, then  we define
$$
L_1\cdots L_{n-2}D_{1}D_{2}:=(L_1\cdots L_{n-2}D_{1}D_{2})_{\pi},
$$
where $\pi\colon Y'\to X$ is an alteration with $Y'$ regular (for the existence of alterations, see \cite{dJo}).

\smallskip

\noindent
(2) If $k$ is finite and $H^0(\O_X)=k$, then we take an algebraic closure $\overline{k}$ of $k$ and define
$$
L_1\cdots L_{n-2}D_{1}D_{2}:=L_{1,\overline{k}}\cdots L_{n-2,\overline{k}}D_{1,\overline{k}}D_{2,\overline{k}},
$$
where we put $X_{\overline{k}}:=X\times_{k}\overline{k}$ and the divisors $L_{i,\overline{k}}$ and $D_{i,\overline{k}}$ are respectively the pull-backs of $L_{i}$ and $D_{i}$ via the projection $X_{\overline{k}}\to X$.
Note that $X_{\overline{k}}$ is normal since $k$ is perfect.

\smallskip

\noindent
(3) If $k$ is finite and $k_{X}:=H^0(\O_X)\neq k$, then $X$ is geometrically integral and geometrically normal over $k_{X}$.
Then we define
$$
L_1\cdots L_{n-2}D_{1}D_{2}:=[k_{X}:k](L_1\cdots L_{n-2}D_{1}D_{2})_{X},
$$
where $(L_1\cdots L_{n-2}D_{1}D_{2})_{X}$ is the intersection number on $X$ over $k_{X}$ defined in (2).
\end{definition}

\begin{proof}[Proof of Theorem~\ref{Mumint}]
We define the multilinear form $Q$ as
$$
Q(L_1,\ldots,L_{n-2},D_1,D_2):=L_1\cdots L_{n-2}D_1D_2.
$$
One can see easily that this is well-defined and satisfies the conditions (i), (ii), (iii) and (iv).
\end{proof}


\bigskip


\begin{thebibliography}{99}

\bibitem{AlTo}
A. Alzati and A. Tortora,
On connected divisors,
Adv.\ Geom.\ \textbf{2} (2002), 243--258.


\bibitem{BHPV}
W. Barth, K. Hulek, C. Peters and A. Van de Ven,
{\em Compact complex surfaces. Second edition},
Ergeb.\ Math.\ Grenzgeb., 3.Folge, A Series of Modern Surveys in Mathematics,
Springer-Verlag Berlin 2004.

\bibitem{Ber}
F. Bernasconi,
Kawamata-Viehweg vanishing fails for log del Pezzo surfaces in characteristic $3$,
J.\ of Pure and Applied Algebra \text{225} (2021), 106727.

\bibitem{BeTa}
F. Bernasconi and H. Tanaka,
On del Pezzo fibrations in positive characteristic,
arXiv:1903.10116.

\bibitem{Bom}
E. Bombieri,
Canonical models of surfaces of general type,
Publ.\ Math.\ IHES \textbf{42} (1973), 171--220.

\bibitem{BLR}
S. Bosch, W. L\"{u}tkebohmert, and M. Raynaud,
{\em N\'{e}ron models},
Ergeb.\ Math.\ Grenzgeb., 3.Folge $\cdot$ Band 21, Springer-Verlag, 1990.

\bibitem{BdFF}
S. Boucksom, T. de Fernex, and C. Favre,
The volume of an isolated singularity,
Duke.\ Math.\ J.\ \textbf{161} (2012), 1455--1520.

\bibitem{CaTa}
P. Cascini and H. Tanaka,
Smooth rational surfaces violating Kawamata--Viehweg vanishing,
Eur.\ J.\ Math.\ \textbf{4} (2018), 162--176.

\bibitem{CeFa}
G. Di Cerbo and A. Fanelli,
Effective Matsusaka's theorem for surfaces in characteristic $p$,
Alg.\ and Number theory \textbf{9} (2015), 1453--1475.


\bibitem{Eke}
T. Ekedahl,
Canonical models of surfaces of general type in positive characteristic,
Publ.\ Math.\ IHES \textbf{67} (1988), 97--144.

\bibitem{Eno}
M. Enokizono, 
An integral version of Zariski decompositions on normal surfaces,
arXiv:2007.06519.

\bibitem{Fra}
P. Francia,
On the base points of the bicanonical system,
in: Problems in the theory of surfaces and their classification (Cortona, 1988), pp. 141--150, Sympos.\ Math., XXXII, Academic Press, London, 1991.

\bibitem{Fuj}
O. Fujino,
{\em Foundations of the minimal model program},
MSJ Memoirs, Math.\ Soc.\ of Japan, Vol.\ 35, 2017.

\bibitem{Fujita}
T. Fujita,
Semipositive line bundles,
J.\ Fac.\ Sci.\ Univ.\ Tokyo Sect.\ IA Math.\ \textbf{30} (1984), 353--378.


\bibitem{FGA}
A. Grothendieck, {\em Fondements de la g\'{e}om\'{e}trie alg\'{e}brique}. [Extraits du S\'{e}minaire Bourbaki, 1957--1962.] Secretariat math\'{e}matique, Paris 1962 ii+205 pp.


\bibitem{GZZ}
Y. Gu, L. Zhang, and Y. Zhang,
Counterexamples to Fujita's conjecture on surfaces in positive characteristic,
arXiv:2002.04584.

\bibitem{Hara}
N. Hara,
A characterization of rational singularities in terms of injectivity of Frobenius maps,
Amer.\ J.\ Math.\ \textbf{120} (1998), 981--996.

\bibitem{Har}
R. Hartshorne,
{\em Algebraic Geometry},
Graduate Texts in Math.\ \textbf{52}, Springer-Verlag, New York-Heidelberg, 1977.


\bibitem{HuLe} 
D. Huybrechts and M. Lehn,
{\em The geometry of moduli spaces of sheaves}, 
Aspects of Mathematics, E31, Friedr.\ Vieweg \& Sohn, Braunschweig, 1997.


\bibitem{dJo}
A. J. de Jong,
Smoothness, semi-stability and alterations,
Publ.\ Math.\ IHES \textbf{83} (1996), 51--93.


\bibitem{KaUe}
T. Katsura and K. Ueno,
On elliptic surfaces in characteristic $p$,
Math.\ Ann.\ \textbf{272} (1985), 291--330.


\bibitem{KaMa}
T. Kawachi and V. Ma\c{s}ek,
Reider-type theorems on normal surfaces,
J.\ Alg.\ Geom. \textbf{7} (1998), 239--249.

\bibitem{KMM}
Y. Kawamata, K. Matsuda and K. Matsuki,
{\em Introduction to the minimal model program},
Algebraic geometry, Sendai, 1985, Adv.\ Stud.\ Pure Math., vol.\ 10,
North-Holland, Amsterdam, 1987, pp. 283--360.


\bibitem{Kon}
K. Konno, 
Chain-connected component decomposition of curves on surfaces,
J. Math. Soc. Japan \textbf{62} (2010), 467--486.

\bibitem{LanII}
A. Langer,
Adjoint linear systems on normal surfaces II,
J. Alg. Geom. \textbf{9} (2000), 71--92.

\bibitem{Lan}
A. Langer,
Bogomolov's inequality for Higgs sheaves in positive characteristic,
Invent.\ math.\  \textbf{199} (2015), 889--920.

\bibitem{LMM}
C. Liedtke, G. Martin, and Y. Matsumoto,
Linearly reductive quotient singularities, arXiv:2102.01067.

\bibitem{Mad}
Z. Maddock,
Regular del Pezzo surfaces with irregularity,
J.\ Alg.\ Geom.\ \textbf{25} (2016), 401--429.

\bibitem{Miy}
Y. Miyaoka, 
On the Mumford-Ramanujam vanishing theorem on a surface,
Journ\'ees de G\'eometrie Alg\'ebrique d'Angers, Juilet 1979/Algebraic Geometry, Angers, 1979, Sijthoff \& Noordhoff, Alphen aan den Rijn,1980, pp. 239--247.

\bibitem{Mor}
A. Moriwaki,
Frobenius pull-back of vector bundles of rank 2 over non-uniruled varieties,
Math.\ Ann.\ \textbf{296} (1993), 441--451.

\bibitem{Muk}
S. Mukai,
Counterexamples to Kodaira's vanishing and Yau's inequality in positive characteristics,
Kyoto J.\ Math.\ \textbf{53} (2013), 515--532.


\bibitem{Mum}
D. Mumford, 
The topology of normal surface singularities of an algebraic surface and a criterion for simplicity,
Publ. Math. IHES. \textbf{9} (1961), 5--22.

\bibitem{Mum2}
D. Mumford, 
Pathologies III,
Amer.\ J.\ Math.\ \textbf{89} (1967), 94--104.

\bibitem{Nam}
M. Namba,
{\em Families of meromorphic functions on compact Riemann surfaces},
Lecture Notes in Math. \textbf{767}, Springer-Verlag, Berlin, 1979. 

\bibitem{Pao}
R. Paoletti,
Free pencils on divisors,
Math.\ Ann.\ \textbf{303} (1995), 109--123.

\bibitem{Ram}
C. P. Ramanujam,
Remarks on the Kodaira vanishing theorem,
Ind. J. of Math. N.S. \textbf{36} (1972), 41--51.

\bibitem{Ray}
M. Raynaud, 
Sp\'{e}cialisation du foncteur de Picard,
Publ.\ Math.\ IHES.\ \textbf{38} (1970), 27--76.

\bibitem{Ray2}
M. Raynaud,
Contre-example au ``vanishing de Kodaira'' sur une surface lisse en caract\'{e}ristique $p>0$,
{\em C. P. Ramanujam - A Tribute},
Springer-Verlag, Berlin-Heidelberg-New York, 1978, pp. 273--278.

\bibitem{Reid}
M. Reid, 
Special linear systems on curves lying on a $K-3$ surface,
J.\ Lond.\ Math.\ Soc.\ \textbf{13} (1976), 454--458.

\bibitem{Rei}
I. Reider,
Vector bundles of rank 2 and linear systems on algebraic surfaces,
Ann. of Math. \textbf{127} (1988), 309--316.

\bibitem{S-D}
B. Saint-Donat,
Projective models of $K3$ surfaces,
Amer.\ J.\ Math.\ \textbf{96} (1974), 602--639.

\bibitem{Sak}
F. Sakai,
Reider-Serrano's method on normal surfaces,
Algebraic Geometry: Proceedings L'Aquila 1998, Lecture Notes in Mathematics, \textbf{1417} Springer-Varlag, New-York 1990, 301--319.

\bibitem{Ser}
F. Serrano,
Extension of morphisms defined on divisors,
Math.\ Ann.\ \textbf{277} (1987), 395--413.

\bibitem{S-B}
N.I. Shepherd-Barron, 
Unstable vector bundles and linear systems on surfaces in characteristic $p$,
Invent.\ Math.\ \textbf{106} (1991), 243--261.

\bibitem{Som}
A. J. Sommese, 
On the adjunction theoretic structure of projective varieties,
Lecture Notes in Mathematics, \textbf{1194} Springer-Varlag, Berlin and New York 1986, 175--213.


\bibitem{Tana}
H. Tanaka, 
The X-method for klt surfaces in positive characteristic,
J.\ Alg.\ Geom. \text{24} (2015), 605--628.


\bibitem{Tan}
H. Tango, 
On $(n-1)$-dimensional projective spaces contained in the Grassmann variety $\Gr(n,1)$,
J.\ Math.\ Kyoto Univ.\ \textbf{14} (1974), 415--460.

\bibitem{Ter}
H. Terakawa,
The $k$-very ampleness on a projective surface in positive characteristic,
Pacific J.\ of Math.\ \textbf{187} (1999), 187--199.

\bibitem{Zha}
Y. Zhang,
The $d$-ampleness on quasi-elliptic surfaces,
arXiv:2103.04268.


\end{thebibliography}
\end{document}